\documentclass[10pt]{scrartcl}
\usepackage[T1]{fontenc}
\usepackage[utf8]{inputenc}

\usepackage[left=2cm, right=2cm]{geometry}

\usepackage{amsmath}
\usepackage{amsfonts}
\usepackage{amssymb}
\usepackage{amsthm}

\numberwithin{equation}{section}

\usepackage{bm}

\usepackage[backend=biber,sorting=nty,style=nature,minnames=3,maxnames=5]{biblatex}
\usepackage{abstract}
\usepackage{hyperref}

\newtheorem{theorem}{Theorem}[section]
\newtheorem{lemma}[theorem]{Lemma}

\theoremstyle{definition}
\newtheorem{problem}[theorem]{Problem}
\newtheorem{definition}[theorem]{Definition}

\newcommand{\abs}[1]{\left\vert #1 \right\vert}
\newcommand{\norm}[1]{\left\Vert #1 \right\Vert}
\newcommand{\opnorm}[1]{\left\vert\kern-0.25ex\left\vert\kern-0.25ex\left\vert #1 \right\vert\kern-0.25ex\right\vert\kern-0.25ex\right\vert}

\newcommand{\scalar}[2]{\left( #1, #2 \right)}

\begin{document}
    \title{Time-continuous strongly conservative space-time finite element methods for the dynamic Biot model}
    \author{Johannes Kraus\thanks{University of Duisburg-Essen, Faculty of Mathematics, Thea-Leymann-Str. 9, 45127 Essen, Germany} \thanks{Corresponding author: \href{mailto:johannes.kraus@uni-due.de}{johannes.kraus@uni-due.de}} \and Maria Lymbery\footnotemark[1] \and Kevin Osthues\footnotemark[1]}
    
    \maketitle
    
    \begin{abstract}
        We consider the dynamic Biot model (see~[Biot,~M.~A. \textit{J. Appl. Phys.}~\textbf{33,} 1482--1498 (1962)]) describing the interaction between fluid flow and solid
        deformation including wave propagation phenomena in both the liquid and solid phases of a saturated porous medium. This model
        couples a hyperbolic equation for momentum balance to a second-order in time dynamic Darcy law and a parabolic equation for the
        balance of mass and is here considered in three-field formulation with the displacement of the elastic matrix, the fluid velocity, and the
        fluid pressure being the physical fields of interest.
        
        A family of variational space-time finite element methods is proposed, which combines a continuous-in-time Galerkin ansatz of
        arbitrary polynomial degree with $H(\mathrm{div})$-conforming approximations of the displacement field, its time derivative, and the flux field--of
        discontinuous Galerkin (DG) type for displacements--with a piecewise polynomial pressure approximation, providing an inf-sup stable strongly
        conservative mixed method in each case.
        We prove error estimates in a combined energy norm in space for the maximum norm in time.
        The theoretical results are confirmed by numerical experiments for different polynomial orders in space and time.
    \end{abstract}
    
    \begin{quote}
        \textbf{Keywords.}
        Dynamic Biot model; poroelasticity; time-continuous space-time Galerkin,
        strongly conservative $H(\mathrm{div})$-conforming discontinuous Galerkin, mixed method; a priori error analysis.
    \end{quote}

    \section{Introduction}
    The theory of poroelasticity dates back to mid of the last century and was most decisively developed by Maurice Anthony Biot in a series of
    pioneering papers, see~\cite{Biot1941general,Biot1955theory,Biot1956theory,Biot1956waves1,Biot1956waves2}. Starting with a quasi-static
    model of consolidation~\cite{Biot1941general}, the theory was generalized from isotropic to anisotropic media~\cite{Biot1955theory}, viscoelastic
    materials~\cite{Biot1956theory}, and extended to dynamic systems~\cite{Biot1956waves1,Biot1956waves2}.
    The models in~\cite{Biot1956waves1,Biot1956waves2,Biot1962acoust,Biot1962Mech} are derived from general principles of nonequilibrium
    thermodynamics and describe the interaction of fluid flow and deformation of (linear, anisotropic) porous media taking into account also wave
    propagation, energy dissipation, and other relaxation effects.
    
    First fundamental mathematical results related to the well-posedness, discretization and stability of the models proposed by Biot
    were addressed, e.g., in~\cite{Carlson1973Lin,Dafermos1968exist,Showalter2000Diff,Zienkiewicz1982Basic,Zienkiewicz1984dynamic,
    Zenisek1984existence,Mielke2013Homo}, for a more comprehensive overview see also~\cite{Cheng2016Poro,Seifert2022Evol}.
    
    The classical area of application of (thermo-)poroelasticity is geomechanics where typical problems reach from evaluating subsidence
    induced by hydrocarbon production, simulation of underground reservoirs, environmental planning to the engineering of safe barriers for radioactive
    waste disposals, see, e.g.,~\cite{Barenblatt1960basic,Zenisek1984finite,Bjornara2016vertically,Michalec2021fully,Gao2022coupled}.
    In recent years, the important observation that biological tissues can be interpreted as porous, permeable and deformable media infiltrated by fluids,
    such as blood and interstitial fluid, cf. e.g.~\cite{BOTH2022114183, Pitre2017}, has ultimately led to the successful employment of poroelasticity
    models in the studies of the biomechanical behavior of different human organs, such as the brain,
    \cite{Guo_etal2018subject-specific,Vardakis2019fluid,Kraus2023Hybr},
    the liver, \cite{Aichele2021fluids},
    the kidney, \cite{Ong2009Deformation},
    which have offered better understanding of various disease states ranging from renal failure to traumatic brain injury to cancer.
    
    The range of physical parameters in the models under consideration, especially in biomedical applications, makes the design and
    analysis of adequate numerical methods a challenging task. This is why significant effort has been put into developing stable and
    parameter-robust methods for the numerical solution of the underlying coupled partial differential algebraic equations over the last decades.
    
    A rigorous stability and convergence analysis of finite element approximations of the two-field formulation of Biot's model of consolidation
    in terms of the displacement field and the fluid pressure has first been presented in~\cite{MuradLoula1992improved,MuradLoula1994stability}.
    The derived a priori error estimates are valid for both semidiscrete and fully discrete formulations, where inf-sup stable finite elements are used
    for space discretization and the backward Euler method is employed for time-stepping. Moreover, a post‐processing technique has been
    introduced to improve the pore-pressure accuracy.
    Numerous works have addressed various aspects of locking-free and (locally) conservative discretization of this class of problems since then,
    more lately, also using generalized and extended finite element, see, e.g.,~\cite{Fries2010extended} and enriched Galerkin methods, see,
    e.g.,~\cite{Lee2018enriched,Lee2023locking}, and to study flow and wave propagation phenomena in fractured heterogeneous porous
    media, see~\cite{Janaki2018enriched}.
    Optimal $L^2$ error estimates for the quasi-static Biot problem have been derived only recently, see~\cite{Wheeler2022optimal}.
    
    Discretizations of the quasi-static Biot problem in classic three-field formulation, using standard continuous Galerkin approximations
    of the displacement field and a mixed approximation of the flux-pressure pair, where the flux variable is introduced to the system via
    the standard Darcy law, have originally been proposed in~\cite{PhillipsWheeler2007coupling1,PhillipsWheeler2007coupling2}.
    The error estimates proved in these works are for both continuous- and discrete-in-time models.
    This approach has also been extended to discontinuous Galerkin approximations of the displacement field,
    see~\cite{PhillipsWheeler2008coupling}, and other nonconforming approximations, e.g., using modified rotated bilinear
    elements~\cite{Yi2013coupling}, or Crouzeix-Raviart elements for the displacements, see~\cite{Hu2017nonconforming}.
    More recently, in~\cite{Hong2018Para}, a family of strongly mass conserving discretizations based on $H(\mathrm{div})$-conforming
    discontinuous Galerkin approximation of the displacement field has been proven to be uniformly stable in properly fitted
    parameter-dependent norms resulting in near best approximation results that are robust with respect to all model and discretization
    parameters (Lam\'{e} parameters, permeability and Biot-Willis parameter, storage coefficient, time step and mesh size).
    Error estimates for the related continuous- and discrete-in-time scheme have been delivered in~\cite{KanschatRiviere2018finite}
    and hybridization techniques considered in~\cite{Kraus2021Unif}. A stabilized conforming method for the Biot problem in three-field
    formulation was proposed in~\cite{Rodrigo2018new}.
    
    The four-field formulation of the quasi-static Biot problem in~\cite{Yi2014convergence} couples two mixed problems and uses a
    mixed finite element method based on the Hellinger-Reissner principle for the mechanics subproblem and a standard mixed finite
    element method for the flow subproblem.
    Optimal a priori error estimates have been proved for both semidiscrete and fully discrete schemes using the Arnold-Winther
    space and the Raviart-Thomas space in the two coupled mixed finite element methods, respectively. The same discretization for
    the same formulation has further been analyzed in~\cite{Lee2016robust} providing error estimates in $L^\infty$ norm in time and
    $L^2$ norm in space, which are robust with respect to the Lam\'{e} parameters and remain valid in the limiting case of vanishing
    constrained storage coefficient.
    
    Several of these techniques have also been transferred to multiple network poroelasticity theory (MPET), with a focus on
    the stability of continuous and discrete models, mass conservation, hybridization, near-best approximation results, uniform
    preconditioning, and splitting schemes, see, e.g.~\cite{HongEtAl2020parameter,Hong2020Para,Kraus2021Unif}.
    Other three-, four-, and five-field models additionally involve a total pressure and their discretization and a priori error analysis
    as well as robust preconditioners have been addressed,
    e.g., in~\cite{Oyarzua2016locking,Lee2017parameter,Kumar2020conservative,Lee2023analysis}.
    
    Most of the works mentioned so far--except some references in the first two paragraphs--propose and analyze numerical methods
    for quasi-static models only. While these have been intensively studied over the years and are quite well understood, far less is known
    about stable numerical solution methods for the dynamic Biot model. However, there are a few early groundbreaking theoretical results
    on the design and error analysis of finite element approximations for the latter class of problems, see,
    e.g.,~\cite{Dupont1973L2,Baker1976error,Wheeler1973Linfty}.
    
    Driven by the interest to model both the dynamics of the solid and fluid phases, we consider a classic three-field formulation
    here. The dynamic model, which is subject to the error analysis presented in this work, describes the dynamics of the solid and fluid
    phases and takes into account also fluid inertia effects. Note that this requires a mass balance~\eqref{eq:dg_dyn_biot-3} based on a flux
    variable introduced via the dynamic Darcy law~\eqref{eq:dg_dyn_biot-2}.\footnote{This is contrary to the model studied in~\cite{Bause2024Conv},
    which uses the standard Darcy flux in the mass balance equation.}
    
    For discretization, we employ here
    $H(\mathrm{div})$-conforming approximations, conforming for fluxes and nonconforming of discontinuous Galerkin type
    for the displacement field, thus following~\cite{Hong2018Para}.
    The finite element spaces for space-discretizations are chosen in such a way that they ensure point-wise conservation of mass,
    as in~\cite{Hong2018Para,Hong2019Cons}.
    The proofs we present here are inspired by and use ideas
    from~\cite{Karakashian2005Conv,Bause2024Conv}.
    
    Hence, the contribution of the present paper is a rigorous convergence analysis of a family of variational space-time finite element methods
    for the dynamic Biot model as presented in~\cite{Biot1956waves1,Biot1956waves2,Biot1962Mech}.
    The continuous-in-time Galerkin-Petrov scheme that we use for time discretization has also been studied in~\cite{Bause2024Conv} but for
    a different hyperbolic-parabolic problem in two-field formulation.
    For the construction of time discretizations with higher regularity, we refer to \cite{Anselmann2020Gale,Bause2022C1co,Becher2021Vari}.
    
    The remainder of the manuscript is organized as follows: In Section~\ref{sec:problem} we state the three-field formulation of the dynamic Biot
    problem under consideration. Section~\ref{sec:preliminaries} summarizes some of the notation used in this paper, recalls quadrature formulas,
    introduces the interpolation operators and some related useful auxiliary results. Section~\ref{sec:dg_problem_formulations} presents the variational
    formulations giving the starting point for the discrete models and the forthcoming analysis. Section~\ref{sec:dg_error_analysis} contains
    a detailed derivation of the error estimates which are finally given in Theorem~\ref{thm:dg_discretization_error}, which is the main
    theoretical result of this work. In Section~\ref{sec:dg_numerical_experiments}, we present some numerical experiments. Lastly, we end our analysis with a short conclusion.

    \section{Problem formulation}
    \label{sec:problem}
    In this paper, we study the dynamic Biot model for a linear elastic, homogeneous, isotropic, porous solid saturated with a slightly compressible
    fluid as proposed in~\cite{Biot1962Mech, Biot1941general}, that is, the following coupled system of partial differential equations
    \begin{subequations}
        \label{eq:dg_dyn_biot}
        \begin{alignat}{2}
            \bar{\rho} \partial_{t}^{2} \bm{u} - 2 \mu \mathrm{div}(\bm{\epsilon}(\bm{u})) - \lambda \nabla \mathrm{div}(\bm{u}) + \alpha \nabla p + \rho_{f} \partial_{t} \bm{w} & = \bm{f} & \qquad & \text{in $ \Omega \times \left(0, T\right] $}, \label{eq:dg_dyn_biot-1} \\
            \rho_{f} \partial_{t}^{2} \bm{u} + \rho_{w} \partial_{t} \bm{w} + \bm{K}^{-1} \bm{w} + \nabla p & = \bm{0} & \qquad & \text{in $ \Omega \times \left(0, T\right] $}, \label{eq:dg_dyn_biot-2} \\
            s_{0} \partial_{t} p + \alpha \mathrm{div}(\partial_{t} \bm{u}) + \mathrm{div}(\bm{w}) & = g & \qquad & \text{in $ \Omega \times \left(0, T\right] $}, \label{eq:dg_dyn_biot-3}
        \end{alignat}
    \end{subequations}
    in a bounded Lipschitz domain $ \Omega \subset \mathbb{R}^{d}, d \in \{2, 3\} $ and $ T > 0 $.
    The physical parameters in the model are
    the total density $ \bar{\rho} = (1 - \phi_{0}) \rho_{s} + \phi_{0} \rho_{f} > 0 $ and
    the effective fluid density $ \rho_{w} \geq \phi_{0}^{-1} \rho_{f} > 0 $, where $ \rho_{s} > 0 $ represents the solid density, $ \rho_{f} > 0 $ the fluid density and
    $ \phi_{0} \in (0, 1) $ the porosity.
    In addition, the Biot-Willis parameter is specified by $ \alpha \in \left[\phi_{0}, 1\right] $, cf.~\cite{Biot1957Thee} and the constrained specific storage coefficient by $ s_{0} > 0 $, cf.~\cite{Biot1962Mech}.
    The material-dependent Lam\'{e} parameters are denoted by $ \lambda > 0 $ and $ \mu > 0 $.
    Furthermore, $ \bm{K} $ represents the symmetric positive definite permeability tensor, $ \bm{f}: \Omega \times [0, T] \to \mathbb{R}^{d} $ the body force density
    and $ g: \Omega \times [0, T] \to \mathbb{R} $ a source term.
    The strain tensor is given by the symmetric part of the gradient of $ \bm{u} $, that is $ \bm{\epsilon}(\bm{u}) = (\nabla \bm{u} + (\nabla \bm{u})^{\top}) / 2 $.
    
    The dynamic Biot model comprises the momentum balance equation \eqref{eq:dg_dyn_biot-1}, the dynamic Darcy law \eqref{eq:dg_dyn_biot-2}
    and the mass balance equation \eqref{eq:dg_dyn_biot-3}, see \cite{Mielke2013Homo}.
    Therein, the unknown physical fields are described by the displacement $ \bm{u} = \bm{u}(\mathbf{x}, t) $, seepage velocity $ \bm{w} = \bm{w}(\mathbf{x}, t) $
    and fluid pressure $ p = p(\mathbf{x}, t) $.
    The model in~\cite{Bause2024Conv} follows from system~\eqref{eq:dg_dyn_biot} by setting the fluid density and herewith the effective fluid density to zero
    resulting in a simplification of the momentum balance equation~\eqref{eq:dg_dyn_biot-1}. The latter model is suitable for applications in which fluid inertia effects
    can be neglected.
    By neglecting the acceleration of the solid and the fluid phases, the dynamic Biot model~\eqref{eq:dg_dyn_biot} reduces to the classical three-field formulation
    of the quasi-static Biot problem~\cite{Hong2018Para}.
    We complete problem~\eqref{eq:dg_dyn_biot} by the following initial and boundary conditions
    \begin{alignat}{4}
        \label{eq:dg_boundary_conditions}
        \begin{aligned}
            \bm{u}(\cdot, 0) & = \bm{u}_{0} & \quad & \text{in $ \Omega $}, \qquad\qquad & 
            \bm{u} & = \bm{0} & \quad & \text{on $ \partial \Omega \times \left(0, T\right] $}, \\
            \partial_{t} \bm{u}(\cdot, 0) & = \bm{v}_{0} & \quad & \text{in $ \Omega $}, \qquad\qquad &
            \bm{w} \cdot \bm{n} & = 0 & \quad & \text{on $ \partial \Omega \times \left(0, T\right] $}, \\
            p(\cdot, 0) & = p_{0} & \quad & \text{in $ \Omega $}.
        \end{aligned}
    \end{alignat}
    Subject to the initial and boundary conditions \eqref{eq:dg_boundary_conditions}, problem~\eqref{eq:dg_dyn_biot} is well-posed which can be shown using Picard's theorem~\cite[Thm.~6.2.1]{Seifert2022Evol}.

    \section{Preliminaries}
    \label{sec:preliminaries}
    
    \subsection{Function spaces}
    In this paper, we use standard notation for Sobolev spaces.
    The $ L^{2} $ inner product is denoted by $ \scalar{\cdot}{\cdot} $.
    For our analysis, we need the Bochner spaces $ L^{2}(J; B) $, $ C(J; B) $, $ C^{m}(J; B), m \in \mathbb{N} $, where $ B $ is a Banach space and $ J \subseteq [0, T] $.
    Throughout the remainder, we mark vector- and matrix-valued quantities in bold.
    
    In view of the proposed time stepping method, we split the time interval $ I := \left(0, T\right] $ into $ N $ subintervals $ I_{n} := \left( t_{n-1}, t_{n} \right] $ of length $ \tau_{n} := t_{n} - t_{n-1} $, where $ 0 = t_{0} < t_{1} < \cdots < t_{N} = T $.
    Moreover, we define the time mesh by $ \mathcal{M}_{\tau} := \{I_{1}, \ldots, I_{N}\} $.
    For a fixed $ k \in \mathbb{N}_{0} $, $n \in \{1,2,\ldots,N\}$, and Banach space $B$, we introduce the space of $B$-valued polynomials in time on $I_{n}$ by
    \begin{align*}
        \mathbb{P}_{k}(I_{n}; B) := \{ f_{\tau} : I_{n} \to B \mid f_{\tau}(t) = \sum_{i = 0}^{k} b_{i} t^{i} \ \forall t \in I_{n}, b_{i} \in B, i = 0, \ldots, k \}.
    \end{align*}
    Additionally, we define the space of globally continuous 
    piecewise $B$-valued polynomials in time
    \begin{align*}
        \mathbb{X}_{\tau}^{k}(B) := \{ f_{\tau} \in C(\overline{I}; B) \mid f_{\tau} \big|_{I_{n}} \in \mathbb{P}_{k}(I_{n}; B) \ \forall I_{n} \in \mathcal{M}_{\tau} \}
    \end{align*}
    and the space of piecewise $B$-valued polynomials in time 
    \begin{align*}
        \mathbb{Y}_{\tau}^{k}(B) := \{ f_{\tau} \in L^{2}(I; B) \mid f_{\tau} \big|_{I_{n}} \in \mathbb{P}_{k}(I_{n}; B) \ \forall I_{n} \in \mathcal{M}_{\tau} \}.
    \end{align*}
    On the time intervals $ I_{n} := \left( t_{n-1}, t_{n} \right] $, we will also make use of the right-sided limit of functions $f$, denoted by
    $ f(t_{n-1}^{+}) := \lim_{t \searrow t_{n-1}} f(t) $.

    \subsection{Quadrature formulas and interpolation operators}
    \label{subsec:quadrature_formulas}
    Here, we review some quadrature formulas and interpolation operators needed throughout the remainder of the paper.
    The $ (k+1) $-point Gauss-Lobatto quadrature formula on $ I_{n} = \left(t_{n-1}, t_{n}\right] $ is given by
    \begin{align*}
        Q_{n}^{\text{GL}}(f) := \tau_{n} \sum_{i = 0}^{k} \hat{\omega}_{i}^{\text{GL}} f\big|_{I_{n}}(t_{n, i}^{\text{GL}}),
    \end{align*}
    where $ t_{n, i}^{\text{GL}} := T_{n}(\hat{t}_{i}^{\text{GL}}), i = 0, \ldots, k $ are the quadrature points on interval $ \overline{I}_{n} $ with quadrature weights $ \hat{\omega}_{i}^{\text{GL}} $.
    The quadrature points $ \{t_{n, i}^{\text{GL}}\}_{i=0}^{k} $ are obtained via the affine linear transformation
    $ T_{n} : \hat{I} \to I_{n}, \ T_{n}(\hat{t}) = \tau_{n} \hat{t} + t_{n-1}, $
   from the Gauss-Lobatto quadrature points $ \hat{t}_{i}^{\text{GL}} $ on reference interval $ \hat{I} := [0, 1] $.
    
    Furthermore, we need the $ k $-point Gauss quadrature formula on $ I_{n} $ defined by
    \begin{align}
        \label{eq:quadrature_G}
        Q_{n}^{\text{G}}(f) := \tau_{n} \sum_{i = 1}^{k} \hat{\omega}_{i}^{\text{G}} f\big|_{I_{n}}(t_{n, i}^{\text{G}}),
    \end{align}
    where $ t_{n, i}^{\text{G}} := T_{n}(\hat{t}_{i}^{\text{G}}) $ and $ \hat{\omega}_{i}^{\text{G}}, i = 1, \ldots, k $ denote the quadrature points and weights, 
    respectively, with $ \{\hat{t}_{i}^{\text{G}}\}_{i=1}^{k} $ being the quadrature points on reference interval $ \hat{I} $.
    Both quadrature formulas are exact for polynomials of degree $ 2k-1 $.
    
    In our analysis, we make use of different Lagrange interpolation operators.
    To this end, we first define local operators on each subinterval $ I_{n} $ via
    \begin{alignat*}{3}
        \mathcal{I}_{\tau, n}^{\text{GL}} & : C(\overline{I}_{n}; B) \to \mathbb{P}_{k}(\overline{I}_{n}; B), \qquad &
        \mathcal{I}_{\tau, n}^{\text{GL}} f(t_{n, i}^{\text{GL}}) & = f(t_{n, i}^{\text{GL}}), \quad & i & = 0, \ldots, k, \\
        \mathcal{I}_{\tau, n}^{\text{G}} & : C(\overline{I}_{n}; B) \to \mathbb{P}_{k-1}(\overline{I}_{n}; B), \qquad &
        \mathcal{I}_{\tau, n}^{\text{G}} f(t_{n, i}^{\text{G}}) & = f(t_{n, i}^{\text{G}}), \quad & i & = 1, \ldots, k, \\
        \mathcal{I}_{\tau, n}^{\text{G}, 0} & : C(\overline{I}_{n}; B) \to \mathbb{P}_{k}(\overline{I}_{n}; B), \qquad &
        \mathcal{I}_{\tau, n}^{\text{G}, 0} f(t_{n, i}^{\text{G}, 0}) & = f(t_{n, i}^{\text{G}, 0}), \quad & i & = 0, \ldots, k,
    \end{alignat*}
    where $ B $ is a Banach space, $ t_{n, 0}^{\text{G}, 0} := t_{n-1} $ and $ t_{n, i}^{\text{G}, 0} := t_{n, i}^{\text{G}} $ for $ i = 1, \ldots, k $.
    With these local interpolation operators at hand, we construct the global interpolation operators
    \begin{alignat*}{3}
        \mathcal{I}_{\tau}^{\text{GL}} & : C(\overline{I}; B) \to \mathbb{X}_{\tau}^{k}(B), \qquad &
        (\mathcal{I}_{\tau}^{\text{GL}} f)\big|_{I_{n}} & = \mathcal{I}_{\tau, n}^{\text{GL}}(f\big|_{I_{n}}), \quad & n & = 1, \ldots, N, \\
        \mathcal{I}_{\tau}^{\text{G}} & : C(\overline{I}; B) \to \mathbb{Y}_{\tau}^{k-1}(B), \qquad &
        (\mathcal{I}_{\tau}^{\text{G}} f)\big|_{I_{n}} & = \mathcal{I}_{\tau, n}^{\text{G}}(f\big|_{I_{n}}), \quad & n & = 1, \ldots, N, \\
        \mathcal{I}_{\tau}^{\text{G}, 0} & : C(\overline{I}; B) \to \mathbb{Y}_{\tau}^{k}(B), \qquad &
        (\mathcal{I}_{\tau}^{\text{G}, 0} f)\big|_{I_{n}} & = \mathcal{I}_{\tau, n}^{\text{G}, 0}(f\big|_{I_{n}}), \quad & n & = 1, \ldots, N.
    \end{alignat*}
    Next, we define the Lagrange interpolation polynomials $ \{L_{n, i}^{\text{G}}\}_{i=1}^{k} $ and $ \{L_{n, i}^{\text{G}, 0}\}_{i=0}^{k} $ corresponding to the points $ \{t_{n, i}^{\text{G}}\}_{i=1}^{k} $ and $ \{t_{n, i}^{\text{G}, 0}\}_{i=0}^{k} $.
    Recall that on each subinterval $ I_{n} $, the Lagrange interpolation operator $ \mathcal{I}_{\tau}^{\text{GL}} $ fulfills
    \begin{align}
        \label{eq:interpolation_error}
        \norm{f - \mathcal{I}_{\tau}^{\text{GL}} f}_{L^{s}(I_{n}; L^{2})}
        & \leq c \tau_{n}^{k+1} \norm{\partial_{t}^{k+1} f}_{L^{s}(I_{n}; L^{2})},
    \end{align}
    with $ s \in \{2, \infty\} $, see, e.g., \cite{Howell1991Deri}.
    In addition, we define the projection $ \Pi_{\tau}^{k-1}: L^{2}(I; L^{2}(\Omega)) \to \mathbb{Y}_{\tau}^{k-1}(L^{2}(\Omega)) $
    with $ \Pi_{\tau}^{k-1} f\big|_{I_{n}} \in \mathbb{P}_{k-1}(I_{n}; L^{2}(\Omega)) $ via the identity
    \begin{gather}
        \label{eq:local_L2_projection}
        \int_{I_{n}} \scalar{\Pi_{\tau}^{k-1} f}{\varphi} \mathrm{d}t = \int_{I_{n}} \scalar{f}{\varphi} \mathrm{d}t \qquad \forall \varphi \in \mathbb{P}_{k-1}(I_{n}; L^{2}(\Omega)) .
    \end{gather}
    Then, the following lemma holds true.
    
    \begin{lemma}[{\cite{Bause2024Conv}}]
        \label{lem:local_L2_projection}
        For all $ f \in \mathbb{P}_{k}(I_{n}; L^{2}(\Omega)) $ and $ n = 1, \ldots, N $, we have
        \begin{alignat*}{2}
            \Pi_{\tau}^{k-1} f(t) & = \mathcal{I}_{\tau, n}^{\textnormal{G}} f(t) & \qquad & \forall t \in I_{n}, \\
            \Pi_{\tau}^{k-1} f(t_{n, i}^{\textnormal{G}}) & = f(t_{n, i}^{\textnormal{G}}), & \qquad & i = 1, \ldots, k.
        \end{alignat*}
    \end{lemma}
    
    Besides that, we will make use of the following $L^{\infty} $-$ L^{2}$ inverse estimate.
    
    \begin{lemma}[{\cite{Karakashian2005Conv}}]
        \label{lem:Linfty_L2_inverse_inequality}
        For all $ f \in \mathbb{P}_{k}(I_{n}; \mathbb{R}) $ and $ n = 1, \ldots, N $, there holds that
        \begin{align}
            \label{eq:Linfty_L2_inverse_inequality}
            \max_{t \in I_{n}} \abs{f(t)} \leq c \tau_{n}^{-1/2} \norm{f}_{L^{2}(I_{n}; \mathbb{R})}.
        \end{align}
    \end{lemma}

    \section{Variational formulations}
    \label{sec:dg_problem_formulations}
    To begin with, we introduce the variable $ \bm{v} := \partial_{t} \bm{u} $ and rewrite problem~\eqref{eq:dg_dyn_biot} as
    \begin{subequations}
        \label{eq:dg_dyn_biot_velo}
        \begin{alignat}{2}
            \bar{\rho} \partial_{t} \bm{v} - 2 \mu \mathrm{div}(\bm{\epsilon}(\bm{u})) - \lambda \nabla \mathrm{div}(\bm{u}) + \alpha \nabla p + \rho_{f} \partial_{t} \bm{w} & = \bm{f} & \qquad & \text{in $ \Omega \times \left(0, T\right] $}, \label{eq:dg_dyn_biot_velo-1} \\
            \partial_{t} \bm{u} - \bm{v} & = \bm{0} & \qquad & \text{in $ \Omega \times \left(0, T\right] $}, \label{eq:dg_dyn_biot_velo-2} \\
            \rho_{f} \partial_{t} \bm{v} + \rho_{w} \partial_{t} \bm{w} + \bm{K}^{-1} \bm{w} + \nabla p & = \bm{0} & \qquad & \text{in $ \Omega \times \left(0, T\right] $}, \label{eq:dg_dyn_biot_velo-3} \\
            s_{0} \partial_{t} p + \alpha \mathrm{div}(\partial_{t} \bm{u}) + \mathrm{div}(\bm{w}) & = g & \qquad & \text{in $ \Omega \times \left(0, T\right] $}. \label{eq:dg_dyn_biot_velo-4}
        \end{alignat}
    \end{subequations}
    In the following, we present a time-continuous as well as a time-discrete variational formulation of problem~\eqref{eq:dg_dyn_biot_velo}.
    The latter will be obtained from application of a continuous Galerkin-Petrov method in time. Strongly conservative space discretization we
    achieve by combining mixed and $H(\mathrm{div})$-conforming standard and discontinuous Galerkin methods.

    \subsection{Space-time variational problem}
    In order to give a space-time variational formulation of problem~\eqref{eq:dg_dyn_biot_velo}, we define the following spaces
    \begin{align*}
        \begin{aligned}
            \bm{U} = \bm{V} & := \bm{H}_{0}^{1}(\Omega)
            = \{ \bm{u} \in \bm{H}^{1}(\Omega) \mid \text{$ \bm{u} = \bm{0} $ on $ \partial \Omega $} \}, \\
            \bm{W} & := \bm{H}_{0}(\mathrm{div}, \Omega)
            = \{ \bm{w} \in \bm{H}(\mathrm{div}, \Omega) \mid \text{$ \bm{w} \cdot \bm{n} = 0 $ on $ \partial \Omega $} \}, \\
            P & := L_{0}^{2}(\Omega)
            = \{ p \in L^{2}(\Omega) \mid \int_{\Omega} p \mathrm{d}\mathbf{x} = 0 \}, \\
            \bm{X} & := \bm{U} \times \bm{V} \times \bm{W} \times P,
        \end{aligned}
    \end{align*}
    and bilinear form $ a(\bm{u}, \bm{\varphi}^{\bm{u}}) := 2 \mu \scalar{\bm{\epsilon}(\bm{u})}{\bm{\epsilon}(\bm{\varphi}^{\bm{u}})} + \lambda \scalar{\mathrm{div}(\bm{u})}{\mathrm{div}(\bm{\varphi}^{\bm{u}})} $ for $ \bm{u}, \bm{\varphi}^{\bm{u}} \in \bm{U} $.
    Next, we pose the time-continuous weak problem:
    
    \begin{problem}
        \label{pr:dg_dyn_biot_cont_weak_form}
        Let $ \bm{f} \in L^{2}(I; \bm{L}^{2}(\Omega)) $, $ g \in L^{2}(I; L^{2}(\Omega)) $ and $ (\bm{u}_{0}, \bm{v}_{0}, \bm{w}_{0}, p_{0}) \in \bm{X} $ be given.
        Find $ (\bm{u}, \bm{v}, \bm{w}, p) \in L^{2}(I; \bm{X}) $ such that
        \begin{align*}
            \bm{u}(\cdot, 0) & = \bm{u}_{0}, &
            \bm{v}(\cdot, 0) & = \bm{v}_{0}, &
            \bm{w}(\cdot, 0) & = \bm{w}_{0}, &
            p(\cdot, 0) & = p_{0},
        \end{align*}
        and
        \begin{align*}
            \int_{I} \scalar{\bar{\rho} \partial_{t} \bm{v} + \rho_{f} \partial_{t} \bm{w}}{\bm{\varphi}^{\bm{u}}}
            + a(\bm{u}, \bm{\varphi}^{\bm{u}})
            - \scalar{\alpha p}{\mathrm{div}(\bm{\varphi}^{\bm{u}})} \mathrm{d}t
            & = \int_{I} \scalar{\bm{f}}{\bm{\varphi}^{\bm{u}}} \mathrm{d}t, \\
            \int_{I} \scalar{\partial_{t} \bm{u} - \bm{v}}{\bm{\varphi}^{\bm{v}}} \mathrm{d}t
            & = 0, \\
            \int_{I} \scalar{\rho_{f} \partial_{t} \bm{v} + \rho_{w} \partial_{t} \bm{w} + \bm{K}^{-1} \bm{w}}{\bm{\varphi}^{\bm{w}}}
            - \scalar{p}{\mathrm{div}(\bm{\varphi}^{\bm{w}})} \mathrm{d}t
            & = 0, \\
            \int_{I} \scalar{s_{0} \partial_{t} p + \alpha \mathrm{div}(\partial_{t} \bm{u}) + \mathrm{div}(\bm{w})}{\varphi^{p}} \mathrm{d}t
            & = \int_{I} \scalar{g}{\varphi^{p}} \mathrm{d}t,
        \end{align*}
        for all $ (\bm{\varphi}^{\bm{u}}, \bm{\varphi}^{\bm{v}}, \bm{\varphi}^{\bm{w}}, \varphi^{p}) \in L^{2}(I; \bm{X}) $.
    \end{problem}
    The well-posedness of Problem~\ref{pr:dg_dyn_biot_cont_weak_form} is proven in \cite{Mielke2013Homo}.

    \subsection{Space-time discretization}
    \label{subsec:dg_discrete_formulation}
    We utilize a continuous Galerkin-Petrov (cGP($ k $)) method of degree $k \in \mathbb{N}$ in time for problem~\eqref{eq:dg_dyn_biot_velo},
    for details see,~e.g.,~\cite{Bause2024Conv}.
    With this time stepping scheme, we end up with a globally continuous solution in time.
    For the spatial discretization, we employ a discontinuous Galerkin method for the displacement $ \bm{u} $.
    To this end, we define the following discrete spaces
    \begin{align*}
        \bm{U}_{h} = \bm{V}_{h} = \bm{W}_{h} & := \{ \bm{u}_{h} \in \bm{H}_{0}(\mathrm{div}, \Omega) \mid \bm{u}_{h}\big|_{K} \in \bm{U}(K) \ \forall K \in \mathcal{T}_{h} \}, \\
        P_{h} & := \{ p_{h} \in L_{0}^{2}(\Omega) \mid p_{h}\big|_{K} \in P(K) \ \forall K \in \mathcal{T}_{h} \}, \\
        \bm{X}_{h} & := \bm{U}_{h} \times \bm{V}_{h} \times \bm{W}_{h} \times P_{h},
    \end{align*}
    where $ \mathcal{T}_{h} $ is a shape-regular triangulation and $ \bm{U}(K) \times P(K) $ are chosen to be $ \text{BDM}_{\ell+1}(K) \times \mathbb{P}_{\ell}(K) $ for $ \ell \geq 0 $.
    The space $ \text{BDM}_{\ell}(K) $ denotes the Brezzi-Douglas-Marini space of degree $ \ell \in \mathbb{N} $. 
    To ensure strong, i.e., pointwise, mass conservation, one has to satisfy the conditions $ \mathrm{div}(\bm{U}_{h}) \subseteq P_{h} $,
    and $ \mathrm{div}(\bm{W}_{h}) \subseteq P_{h} $, cf.~\cite{Cockburn2007Anot}, which hold for our choice(s) of spaces.
   
    Let $ \{\!\!\{\cdot\}\!\!\} $ be the average and $ [\![ \cdot ]\!] $ the jump defined by
    \begin{align*}
        \{\!\!\{\bm{\tau}\}\!\!\} & := \frac{1}{2} \left( \bm{\tau}\big|_{\partial K_{1} \cap e} \bm{n}_{1} - \bm{\tau}\big|_{\partial K_{2} \cap e} \bm{n}_{2} \right), &
        [\![ \bm{y} ]\!] & := \bm{y}\big|_{\partial K_{1} \cap e} - \bm{y}\big|_{\partial K_{2} \cap e},
    \end{align*}
    for any interior facet $ e := K_{1} \cap K_{2} $ and
    \begin{align*}
        \{\!\!\{\bm{\tau}\}\!\!\} & := \bm{\tau}\big|_{e} \bm{n}, &
        [\![ \bm{y} ]\!] & := \bm{y}\big|_{e},
    \end{align*}
    for any boundary facet $ e $, where $ \bm{\tau} $ denotes a matrix-valued and $ \bm{y} $ a vector-valued function, see e.g.~\cite{Hong2018Para}.
    Additionally, we specify the tangential component by $ \mathrm{tang}(\bm{z}) := \bm{z} - (\bm{z} \cdot \bm{n}) \bm{n} $.
    Next, we introduce the following mesh-dependent bilinear form
    \begin{align*}
        a_{h}(\bm{u}_{\tau, h}, \bm{\varphi}_{\tau, h}^{\bm{u}}) & := \sum_{K \in \mathcal{T}_{h}} 2 \mu \int_{K} \bm{\epsilon}(\bm{u}_{\tau, h}) : \bm{\epsilon}(\bm{\varphi}_{\tau, h}^{\bm{u}}) \mathrm{d}\mathbf{x}
        - \sum_{e \in \mathcal{E}_{h}} 2 \mu \int_{e} \{\!\!\{\bm{\epsilon}(\bm{u}_{\tau, h})\}\!\!\} \cdot [\![ \mathrm{tang}(\bm{\varphi}_{\tau, h}^{\bm{u}}) ]\!] \mathrm{d}s \\
        & \quad - \sum_{e \in \mathcal{E}_{h}} 2 \mu \int_{e} \{\!\!\{\bm{\epsilon}(\bm{\varphi}_{\tau, h}^{\bm{u}})\}\!\!\} \cdot [\![ \mathrm{tang}(\bm{u}_{\tau, h}) ]\!] \mathrm{d}s
        + \sum_{e \in \mathcal{E}_{h}} 2 \mu \eta h_{e}^{-1} \int_{e} [\![ \mathrm{tang}(\bm{u}_{\tau, h}) ]\!] \cdot [\![ \mathrm{tang}(\bm{\varphi}_{\tau, h}^{\bm{u}}) ]\!] \mathrm{d}s \\
        & \quad + \lambda \scalar{\mathrm{div}(\bm{u}_{\tau, h})}{\mathrm{div}(\bm{\varphi}_{\tau, h}^{\bm{u}})},
    \end{align*}
    where $ h_{e} $ is the diameter of the facet $ e $, the set $ \mathcal{E}_{h} $ contains all facets, and $ \eta $ is a stabilization parameter, cf.~\cite{Hong2018Para}.
    The proposed space-time finite element method then reads:
    
    \begin{problem}
        \label{pr:dg_dyn_biot_disc_weak_form}
        Let $ k, \ell \in \mathbb{N} $ and
        \begin{align*}
            \bm{u}_{\tau, h}^{n-1} & := \bm{u}_{\tau, h}\big|_{I_{n-1}}(t_{n-1}), &
            \bm{v}_{\tau, h}^{n-1} & := \bm{v}_{\tau, h}\big|_{I_{n-1}}(t_{n-1}), &
            \bm{w}_{\tau, h}^{n-1} & := \bm{w}_{\tau, h}\big|_{I_{n-1}}(t_{n-1}), &
            p_{\tau, h}^{n-1} & := p_{\tau, h}\big|_{I_{n-1}}(t_{n-1}),
        \end{align*}
        be given where
        $ (\bm{u}_{\tau, h}^{0}, \bm{v}_{\tau, h}^{0}, \bm{w}_{\tau, h}^{0}, p_{\tau, h}^{0}) := (\bm{u}_{0, h}, \bm{v}_{0, h}, \bm{w}_{0, h}, p_{0, h}) $.
        Find $ (\bm{u}_{\tau, h}, \bm{v}_{\tau, h}, \bm{w}_{\tau, h}, p_{\tau, h}) \in \mathbb{P}_{k}(I_{n}; \bm{X}_{h}) $ such that
        $ (\bm{u}_{\tau, h}(t_{n-1}), \bm{v}_{\tau, h}(t_{n-1}), \bm{w}_{\tau, h}(t_{n-1}), p_{\tau, h}(t_{n-1})) = (\bm{u}_{\tau, h}^{n-1}, \bm{v}_{\tau, h}^{n-1}, \bm{w}_{\tau, h}^{n-1}, p_{\tau, h}^{n-1}) $
        and
        \begin{align*}
            \int_{I_{n}} \scalar{\bar{\rho} \partial_{t} \bm{v}_{\tau, h} + \rho_{f} \partial_{t} \bm{w}_{\tau, h}}{\bm{\varphi}_{\tau, h}^{\bm{u}}}
            + a_{h}(\bm{u}_{\tau, h}, \bm{\varphi}_{\tau, h}^{\bm{u}})
            - \scalar{\alpha p_{\tau, h}}{\mathrm{div}(\bm{\varphi}_{\tau, h}^{\bm{u}})} \mathrm{d}t
            & = \int_{I_{n}} \scalar{\pmb{\mathcal{I}}_{\tau}^{\text{GL}} \bm{f}}{\bm{\varphi}_{\tau, h}^{\bm{u}}} \mathrm{d}t, \\
            \int_{I_{n}} \scalar{\partial_{t} \bm{u}_{\tau, h} - \bm{v}_{\tau, h}}{\bm{\varphi}_{\tau, h}^{\bm{v}}} \mathrm{d}t
            & = 0, \\
            \int_{I_{n}} \scalar{\rho_{f} \partial_{t} \bm{v}_{\tau, h} + \rho_{w} \partial_{t} \bm{w}_{\tau, h} + \bm{K}^{-1} \bm{w}_{\tau, h}}{\bm{\varphi}_{\tau, h}^{\bm{w}}}
            - \scalar{p_{\tau, h}}{\mathrm{div}(\bm{\varphi}_{\tau, h}^{\bm{w}})} \mathrm{d}t
            & = 0, \\
            \int_{I_{n}} \scalar{s_{0} \partial_{t} p_{\tau, h} + \alpha \mathrm{div}(\partial_{t} \bm{u}_{\tau, h}) + \mathrm{div}(\bm{w}_{\tau, h})}{\varphi_{\tau, h}^{p}} \mathrm{d}t
            & = \int_{I_{n}} \scalar{\mathcal{I}_{\tau}^{\text{GL}} g}{\varphi_{\tau, h}^{p}} \mathrm{d}t,
        \end{align*}
        for all $ (\bm{\varphi}_{\tau, h}^{\bm{u}}, \bm{\varphi}_{\tau, h}^{\bm{v}}, \bm{\varphi}_{\tau, h}^{\bm{w}}, \varphi_{\tau, h}^{p}) \in \mathbb{P}_{k-1}(I_{n}; \bm{X}_{h}) $.
    \end{problem}

    \section{Error analysis}
    \label{sec:dg_error_analysis}
    In this section, we derive estimates for the discretization error of the space-time variational method presented in Problem~\ref{pr:dg_dyn_biot_disc_weak_form}.
    The presented analysis is guided by the works~\cite{Karakashian2005Conv, Bause2024Conv} but thereby details on the considered dynamic Biot model as well
    as its strongly conservative space discretization involving a mixed method for the flow subproblem and $H(\mathrm{div})$-conforming discontinuous Galerkin
    method for the mechanics subproblem.
    
    \subsection{Norms and projection operators}
    For the error analysis, we define the following norms
    \begin{align*}
        \norm{\bm{u}_{h}}_{\text{DG}}^{2} & := \sum_{K \in \mathcal{T}_{h}} \norm{\bm{\epsilon}(\bm{u}_{h})}_{\bm{L}^{2}(K)}^{2}
        + \sum_{e \in \mathcal{E}_{h}} h_{e}^{-1} \norm{[\![ \mathrm{tang}(\bm{u}_{h}) ]\!]}_{\bm{L}^{2}(e)}^{2}
        + \sum_{K \in \mathcal{T}_{h}} h_{K}^{2} \abs{\bm{u}_{h}}_{\bm{H}^{2}(K)}^{2}, \\
        \norm{\bm{u}_{h}}_{\bm{U}_{h}}^{2} & := \norm{\bm{u}_{h}}_{\text{DG}}^{2} + \lambda \norm{\mathrm{div}(\bm{u}_{h})}_{L^{2}(\Omega)}^{2},
    \end{align*}
    where $ h_{K} $ represents the mesh size of cell $ K $.
    According to \cite{Hong2018Para}, the bilinear form $ a_{h}(\cdot, \cdot) $ satisfies
    \begin{alignat*}{2}
        \abs{a_{h}(\bm{u}_{h}, \bm{\varphi}_{h}^{\bm{u}})}
        & \leq c \norm{\bm{u}_{h} }_{\bm{U}_{h}} \norm{\bm{\varphi}_{h}^{\bm{u}}}_{\bm{U}_{h}}
        & \qquad & \forall \bm{u}_{h} , \bm{\varphi}_{h}^{\bm{u}} \in \bm{H}^{2}(\mathcal{T}_{h}), \\
        c \norm{\bm{u}_{h} }_{\bm{U}_{h}}^{2}
        & \leq a_{h}(\bm{u}_{h} , \bm{u}_{h} )
        & \qquad & \forall \bm{u}_{h}  \in \bm{U}_{h}.
    \end{alignat*}
    Thereby, we have used $ \bm{H}^{2}(\mathcal{T}_{h}) := \{ \bm{u} \in \bm{L}^{2}(\Omega) \mid \bm{u}\big|_{K} \in \bm{H}^{2}(K) \ \forall K \in \mathcal{T}_{h} \} $.
    We define the following projection operators
    \begin{subequations}
        \begin{alignat}{3}
            \bm{\mathcal{R}}_{h} & : \bm{H}(\mathrm{div}, \Omega) \to \bm{U}_{h}, \qquad &
            a_{h}(\bm{\mathcal{R}}_{h} \bm{y}, \bm{\varphi}_{h}^{\bm{u}}) & = a_{h}(\bm{y}, \bm{\varphi}_{h}^{\bm{u}}) & \quad & \forall \bm{\varphi}_{h}^{\bm{u}} \in \bm{U}_{h}, \label{eq:dg_projection_Rh} \\
            \bm{\mathcal{A}}_{h} & : \bm{H}(\mathrm{div}, \Omega) \to \bm{U}_{h}, \qquad &
            \scalar{\bm{\mathcal{A}}_{h} \bm{y}}{\bm{\varphi}_{h}^{\bm{u}}} & = a_{h}(\bm{y}, \bm{\varphi}_{h}^{\bm{u}}) & \quad & \forall \bm{\varphi}_{h}^{\bm{u}} \in \bm{U}_{h}, \label{eq:dg_projection_Ah} \\
            \bm{\mathcal{S}}_{h} & : \bm{W} \to \bm{W}_{h}, \qquad &
            \scalar{\mathrm{div}(\bm{\mathcal{S}}_{h} \bm{y})}{\varphi_{h}^{p}} & = \scalar{\mathrm{div}(\bm{y})}{\varphi_{h}^{p}} & \qquad & \forall \varphi_{h}^{p} \in P_{h}, \label{eq:dg_projection_Sh} \\
            \mathcal{P}_{h} & : P \to P_{h}, \qquad &
            \scalar{\mathcal{P}_{h} y}{\varphi_{h}^{p}} & = \scalar{y}{\varphi_{h}^{p}} & \quad & \forall \varphi_{h}^{p} \in P_{h}. \label{eq:dg_projection_Ph}
        \end{alignat}
    \end{subequations}
    
    Some error estimates for these projection operators are summarized in the following lemma, cf.~\cite{Hong2016robust,Boffi2013Mixed,Monk2003finite}.
    
    \begin{lemma}
        \label{lem:dg_projection_errors}
        The following projection estimates are fulfilled:
        \begin{align*}
            \norm{\bm{v} - \bm{\mathcal{R}}_{h} \bm{v}}_{\bm{L}^{2}(\Omega)}
            & \leq c h^{\ell+1} \norm{\bm{v}}_{\bm{H}^{\ell+2}(\Omega)}, &
            \norm{\bm{w} - \bm{\mathcal{S}}_{h} \bm{w}}_{\bm{L}^{2}(\Omega)}
            & \leq c h^{\ell+1} \norm{\bm{w}}_{\bm{H}^{\ell+1}(\Omega)}, \\
            \norm{\bm{v} - \bm{\mathcal{R}}_{h} \bm{v}}_{\bm{U}_{h}}
            & \leq c h^{\ell+1} \norm{\bm{v}}_{\bm{H}^{\ell+2}(\Omega)}, &
            \norm{p - \mathcal{P}_{h} p}_{L^{2}(\Omega)}
            & \leq c h^{\ell + 1} \norm{p}_{H^{\ell+1}(\Omega)}.
        \end{align*}
    \end{lemma}
    
    \begin{proof}
    The estimate for $\norm{\bm{v} - \bm{\mathcal{R}}_{h} \bm{v}}_{\bm{U}_{h}}$ is obtained by setting $p=\lambda \mathrm{div} \bm{u}$ and
    $p_h=\lambda \mathrm{div} \bm{u}_h$ in~\cite[Theorem~1]{Hong2016robust}, identifying $\bm{u}$ and $\bm{u}_h$
    with $\bm{v}$ and $\bm{\mathcal{R}}_{h} \bm{v}$, respectively, and then using $\text{BDM}_{\ell+1}$ interpolation error estimates.
    The corresponding suboptimal $L^2$ error estimate, which suffices our purpose later on, is a consequence of 
    $$
    \norm{\bm{w}}_{L^2(\Omega)} \le c \norm{\bm{w}}_{\text{DG}} \le c \norm{\bm{w}}_{\bm{U}_{h}}, \quad \forall \bm{w}  \in \bm{H}_{0}(\mathrm{div}, \Omega),
    $$
    which follows from Poincar\'{e} inequality.
    The other estimates follow from standard $\text{BDM}_{\ell+1}(K)$ and $\mathbb{P}_{\ell}(K)$ interpolation error estimates for functions in $H^{\ell+1}(\Omega)$,
    using the properties of the operators $\bm{\mathcal{S}}_{h}$ and $\mathcal{P}_{h}$, cf.~\cite{Boffi2013Mixed,Monk2003finite}.
    \end{proof}
    
    With these projection operators at hand, we can construct an approximation of the solution $ (\bm{u}, \bm{v}, \bm{w}, p) $ based on the ideas of \cite{Karakashian2005Conv, Bause2024Conv}.
    
    \begin{definition}
        \label{def:dg_definition_z}
        For $ t = 0 $, we specify
        \begin{subequations}
            \label{eq:dg_definition_z}
            \begin{align}
                \label{eq:dg_definition_z-1}
                \bm{z}_{1}(0) & := \bm{\mathcal{R}}_{h} \bm{u}_{0}, &
                \bm{z}_{2}(0) & := \bm{\mathcal{R}}_{h} \bm{v}_{0}, &
                \bm{z}_{3}(0) & := \bm{\mathcal{S}}_{h} \bm{w}_{0}, &
                z_{4}(0) & := \mathcal{P}_{h} p_{0},
            \end{align}
            and determine the following polynomials on $ I_{n} $
            \begin{align}
                \label{eq:dg_definition_z-2}
                \begin{aligned}
                    \bm{z}_{1}\big|_{I_{n}} & := \bm{\mathcal{I}}_{\tau}^{\text{GL}}\left( \int_{t_{n-1}}^{t} \bm{z}_{2}(s) \mathrm{d}s + \bm{\mathcal{R}}_{h} \bm{u}(t_{n-1}) \right), & \qquad
                    \bm{z}_{2}\big|_{I_{n}} & := \bm{\mathcal{I}}_{\tau}^{\text{GL}}(\bm{\mathcal{R}}_{h} \partial_{t} \bm{u}), \\
                    \bm{z}_{3}\big|_{I_{n}} & := \bm{\mathcal{I}}_{\tau}^{\text{GL}} (\bm{\mathcal{S}}_{h} \bm{w}), & \qquad
                    z_{4}\big|_{I_{n}} & := \mathcal{I}_{\tau}^{\text{GL}} (\mathcal{P}_{h} p),
                \end{aligned}
            \end{align}
        \end{subequations}
        for $ n = 1, \ldots, N $.
    \end{definition}
    
    The following statement has been proven in~\cite{Karakashian2005Conv}, see also~\cite{Bause2024Conv}.
    
    \begin{lemma}
        \label{lem:dg_z_property}
        For $ (\bm{z}_{1}, \bm{z}_{2}) $ defined by \eqref{eq:dg_definition_z}, there holds
        \begin{align*}
            \int_{I_{n}} \scalar{\partial_{t} \bm{z}_{1}}{\bm{\varphi}_{\tau, h}^{\bm{v}}} \mathrm{d}t
            = \int_{I_{n}} \scalar{\bm{z}_{2}}{\bm{\varphi}_{\tau, h}^{\bm{v}}} \mathrm{d}t
            \quad \forall \bm{\varphi}_{\tau, h}^{\bm{v}} \in \mathbb{P}_{k-1}(I_{n}; \bm{V}_{h}).
        \end{align*}
    \end{lemma}
    
    A useful implication is given in the next lemma.
    
    \begin{lemma}[{\cite[Lem.~3.9]{Bause2024Conv}}]
        \label{lem:dg_z_property_pointwise}
        Let $ \bm{x}_{\tau, h} \in \mathbb{P}_{k}(I_{n}; \bm{U}_{h}) $ and $ \bm{y}_{\tau, h} \in \mathbb{P}_{k}(I_{n}; \bm{V}_{h}) $ satisfy
        \begin{align*}
            \int_{I_{n}} \scalar{\partial_{t} \bm{x}_{\tau, h} - \bm{y}_{\tau, h}}{\bm{\varphi}_{\tau, h}^{\bm{v}}} \mathrm{d}t = 0
            \quad \forall \bm{\varphi}_{\tau, h}^{\bm{v}} \in \mathbb{P}_{k-1}(I_{n}; \bm{V}_{h}).
        \end{align*}
        Then, it follows
        \begin{align*}
            \partial_{t} \bm{x}_{\tau, h}(t_{n, i}^{\textnormal{G}}) = \bm{y}_{\tau, h}(t_{n, i}^{\textnormal{G}})
            \quad i = 1, \ldots, k, \, n = 1, \ldots, N.
        \end{align*}
    \end{lemma}
    
    Finally, Lemma~\ref{lem:dg_projection_errors} and the interpolation error estimate~\eqref{eq:interpolation_error} together
    imply the following estimates.
    
    \begin{lemma}
        \label{lem:dg_projection_errors_z}
        For $ s \in \{2, \infty\} $, we have that
        \begin{align*}
            \norm{\bm{u} - \bm{z}_{1}}_{L^{s}(I_{n}; \bm{L}^{2})} & \leq c \left( \tau_{n}^{k+1} \mathcal{C}_{t, s}^{n, 1} + h^{\ell+1} \mathcal{C}_{\mathbf{x}, s}^{n, 1} \right), &
            \norm{\bm{u} - \bm{z}_{1}}_{L^{s}(I_{n}; \bm{U}_{h})} & \leq c \left( \tau_{n}^{k+1} \mathcal{C}_{t, s}^{n, 2} + h^{\ell+1} \mathcal{C}_{\mathbf{x}, s}^{n, 2} \right), \\
            \norm{\bm{v} - \bm{z}_{2}}_{L^{s}(I_{n}; \bm{L}^{2})} & \leq c \left( \tau_{n}^{k+1} \mathcal{C}_{t, s}^{n, 3} + h^{\ell+1} \mathcal{C}_{\mathbf{x}, s}^{n, 3} \right), &
            \norm{\bm{\mathcal{R}}_{h} \bm{v} - \bm{z}_{2}}_{L^{s}(I_{n}; \bm{U}_{h})} & \leq c \left( \tau_{n}^{k+1} \mathcal{C}_{t, s}^{n, 4} + h^{\ell+1} \mathcal{C}_{\mathbf{x}, s}^{n, 4} \right), \\
            \norm{\bm{w} - \bm{z}_{3}}_{L^{s}(I_{n}; \bm{L}^{2})} & \leq c \left( \tau_{n}^{k+1} \mathcal{C}_{t, s}^{n, 5} + h^{\ell+1} \mathcal{C}_{\mathbf{x}, s}^{n, 5} \right), &
            \norm{p - z_{4}}_{L^{s}(I_{n}; L^{2})} & \leq c \left( \tau_{n}^{k+1} \mathcal{C}_{t, s}^{n, 6} + h^{\ell+1} \mathcal{C}_{\mathbf{x}, s}^{n, 6} \right),
        \end{align*}
        are valid, where
        \begin{align*}
            \mathcal{C}_{t, s}^{n, 1} & := \norm{\partial_{t}^{k+1} \bm{u}}_{L^{s}(I_{n}; \bm{L}^{2})} + \tau_{n} \mathcal{C}_{t, s}^{n, 3}, &
            \mathcal{C}_{\mathbf{x}, s}^{n, 1} & := \norm{\bm{u}}_{L^{s}(I_{n}; \bm{H}^{\ell+2})} + \tau_{n} \mathcal{C}_{\mathbf{x}, s}^{n, 3}, \\
            \mathcal{C}_{t, s}^{n, 2} & := \norm{\partial_{t}^{k+1} \bm{u}}_{L^{s}(I_{n}; \bm{U}_{h})} + \tau_{n} \mathcal{C}_{t, s}^{n, 4}, &
            \mathcal{C}_{\mathbf{x}, s}^{n, 2} & := \norm{\bm{u}}_{L^{s}(I_{n}; \bm{H}^{\ell+2})} + \tau_{n} \mathcal{C}_{\mathbf{x}, s}^{n, 4}, \\
            \mathcal{C}_{t, s}^{n, 3} & := \norm{\partial_{t}^{k+2} \bm{u}}_{L^{s}(I_{n}; \bm{L}^{2})}, &
            \mathcal{C}_{\mathbf{x}, s}^{n, 3} & := \norm{\partial_{t} \bm{u}}_{L^{s}(I_{n}; \bm{H}^{\ell+2})} + \tau_{n} \norm{\partial_{t}^{2} \bm{u}}_{L^{s}(I_{n}; \bm{H}^{\ell+2})}, \\
            \mathcal{C}_{t, s}^{n, 4} & := \norm{\partial_{t}^{k+2} \bm{u}}_{L^{s}(I_{n}; \bm{U}_{h})}, &
            \mathcal{C}_{\mathbf{x}, s}^{n, 4} & := \norm{\partial_{t} \bm{u}}_{L^{s}(I_{n}; \bm{H}^{\ell+2})}, \\
            \mathcal{C}_{t, s}^{n, 5} & := \norm{\partial_{t}^{k+1} \bm{w}}_{L^{s}(I_{n}; \bm{L}^{2})}, &
            \mathcal{C}_{\mathbf{x}, s}^{n, 5} & := \norm{\bm{w}}_{L^{s}(I_{n}; \bm{H}^{\ell+1})} + \tau_{n} \norm{\partial_{t} \bm{w}}_{L^{s}(I_{n}; \bm{H}^{\ell+1})}, \\
            \mathcal{C}_{t, s}^{n, 6} & := \norm{\partial_{t}^{k+1} p}_{L^{s}(I_{n}; L^{2})}, &
            \mathcal{C}_{\mathbf{x}, s}^{n, 6} & := \norm{p}_{L^{s}(I_{n}; H^{\ell+1})} + \tau_{n} \norm{\partial_{t} p}_{L^{s}(I_{n}; H^{\ell+1})}.
        \end{align*}
    \end{lemma}
    
    \begin{proof}
        The proof of these statements follows the lines of~\cite[Lem.~4.1]{Bause2024Conv}.
    \end{proof}
    
    To simplify the notation, we introduce the symmetric positive definite matrix
    \begin{align*}
        \bm{M}_{\rho} :=
        \begin{pmatrix}
            \bar{\rho} \bm{I} & \rho_{f} \bm{I} \\
            \rho_{f} \bm{I} & \rho_{w} \bm{I}
        \end{pmatrix},
    \end{align*}
    cf.~\cite{Mielke2013Homo}, and define the energy norm
    \begin{align*}
        \opnorm{(\bm{u}, \bm{v}, \bm{w}, p)}_{e}^{2}
        := \frac{1}{2} \, a_{h}(\bm{u}, \bm{u})
        + \frac{1}{2} \scalar{\bm{M}_{\rho}
            \begin{pmatrix}
                \bm{v} \\
                \bm{w}
            \end{pmatrix}
        }{
            \begin{pmatrix}
                \bm{v} \\
                \bm{w}
            \end{pmatrix}
        }
        + \frac{s_{0}}{2} \norm{p}_{L^{2}(\Omega)}^{2}.
    \end{align*}
    Due to the coercivity and boundedness of $ a_{h}(\cdot, \cdot) $, this norm is equivalent to
    \begin{align*}
        \opnorm{(\bm{u}, \bm{v}, \bm{w}, p)}^{2}
        := \norm{\bm{u}}_{\bm{U}_{h}}^{2}
        + \norm{\bm{v}}_{\bm{L}^{2}(\Omega)}^{2}
        + \norm{\bm{w}}_{\bm{L}^{2}(\Omega)}^{2}
        + \norm{p}_{L^{2}(\Omega)}^{2}.
    \end{align*}
    Moreover, we will make use of the Bochner norm
    \begin{align*}
        \opnorm{(\bm{u}, \bm{v}, \bm{w}, p)}_{L^{2}(I_{n}; \bm{L}^{2})}^{2}
        := \int_{I_{n}} \opnorm{(\bm{u}, \bm{v}, \bm{w}, p)}^{2} \mathrm{d}t.
    \end{align*}

    \subsection{Error estimates}
    Here, we adapt the strategy used in \cite{Karakashian2005Conv, Bause2024Conv} to the dynamic Biot problem.
    Subsequently, we assume that all occurring norms in the analysis are finite and that the estimates
    \begin{align}
        \label{eq:dg_assumptions}
        \begin{aligned}
            \norm{\bm{\mathcal{R}}_{h} \bm{u}_{0} - \bm{u}_{0, h}}_{\bm{U}_{h}}
            & \leq c h^{\ell+1} \norm{\bm{u}_{0}}_{\bm{H}^{\ell+2}(\Omega)}, &
            \norm{\bm{\mathcal{S}}_{h} \bm{w}_{0} - \bm{w}_{0, h}}_{\bm{L}^{2}(\Omega)}
            & \leq c h^{\ell+1} \norm{\bm{w}_{0}}_{\bm{H}^{\ell+1}(\Omega)}, \\
            \norm{\bm{\mathcal{R}}_{h} \bm{v}_{0} - \bm{v}_{0, h}}_{\bm{L}^{2}(\Omega)}
            & \leq c h^{\ell+1} \norm{\bm{v}_{0}}_{\bm{H}^{\ell+2}(\Omega)}, &
            \norm{\mathcal{P}_{h} p_{0} - p_{0, h}}_{L^{2}(\Omega)}
            & \leq c h^{\ell+1} \norm{p_{0}}_{H^{\ell+1}(\Omega)},
        \end{aligned}
    \end{align}
    hold true.
    In a first step, we split the discretization error into two parts, i.e.,
    \begin{align}
        \label{eq:dg_eta_e_splitting}
        \begin{pmatrix}
            \bm{u} - \bm{u}_{\tau, h} \\
            \bm{v} - \bm{v}_{\tau, h} \\
            \bm{w} - \bm{w}_{\tau, h} \\
            p - p_{\tau, h}
        \end{pmatrix}
        & =
        \begin{pmatrix}
            \bm{u} - \bm{z}_{1} \\
            \bm{v} - \bm{z}_{2} \\
            \bm{w} - \bm{z}_{3} \\
            p - z_{4}
        \end{pmatrix}
        +
        \begin{pmatrix}
            \bm{z}_{1} - \bm{u}_{\tau, h} \\
            \bm{z}_{2} - \bm{v}_{\tau, h} \\
            \bm{z}_{3} - \bm{w}_{\tau, h} \\
            z_{4} - p_{\tau, h}
        \end{pmatrix}
        =:
        \begin{pmatrix}
            \bm{\eta}^{\bm{u}} \\
            \bm{\eta}^{\bm{v}} \\
            \bm{\eta}^{\bm{w}} \\
            \eta^{p}
        \end{pmatrix}
        +
        \begin{pmatrix}
            \bm{e}_{\tau, h}^{\bm{u}} \\
            \bm{e}_{\tau, h}^{\bm{v}} \\
            \bm{e}_{\tau, h}^{\bm{w}} \\
            e_{\tau, h}^{p} \\
        \end{pmatrix}
        =: \bm{\eta}^{\bm{x}} + \bm{e}_{\tau, h}^{\bm{x}},
    \end{align}
    where $ (\bm{z}_{1}, \bm{z}_{2}, \bm{z}_{3}, z_{4}) $ are the piecewise polynomials from Definition~\ref{def:dg_definition_z}.
    For this splitting, analogously to~\cite[Lem.~4.2]{Bause2024Conv}, one obtains the following result.
    
    \begin{lemma}[{cf.~\cite{Bause2024Conv}}]
        \label{lem:dg_variational_problem}
        Let
        \begin{align*}
            \bm{T}_{\text{I}}^{n} & := \bm{f} - \bm{\mathcal{I}}_{\tau}^{\textnormal{GL}} \bm{f}, & \qquad
            \bm{T}_{\text{II}}^{n} & := \bar{\rho} \partial_{t}^{2} \bm{u} - \bar{\rho} \partial_{t} \bm{z}_{2}, & \qquad
            \bm{T}_{\text{III}}^{n} & := \bm{\mathcal{I}}_{\tau}^{\textnormal{GL}} \left( \int_{t_{n-1}}^{t} \partial_{t} \bm{u} - \bm{\mathcal{I}}_{\tau}^{\textnormal{GL}} \partial_{t} \bm{u} \mathrm{d}s \right), \\
            \bm{T}_{\text{IV}}^{n} & := \bm{\mathcal{I}}_{\tau}^{\textnormal{GL}} \bm{u} - \bm{u}, & \qquad
            T_{\text{V}}^{n} & := g - \mathcal{I}_{\tau}^{\textnormal{GL}} g.
        \end{align*}
        Then we have the following variational problem for the components of $ \bm{e}_{\tau, h}^{\bm{x}} $
        \begin{subequations}
            \label{eq:dg_var_eqn}
            \begin{align}
                \begin{split}
                    \label{eq:dg_var_eqn-1}
                    & \int_{I_{n}} \scalar{\bar{\rho} \partial_{t} \bm{e}_{\tau, h}^{\bm{v}} + \rho_{f} \partial_{t} \bm{e}_{\tau, h}^{\bm{w}}}{\bm{\varphi}_{\tau, h}^{\bm{u}}}
                    + \scalar{\bm{\mathcal{A}}_{h} \bm{e}_{\tau, h}^{\bm{u}}}{\bm{\varphi}_{\tau, h}^{\bm{u}}}
                    - \scalar{\alpha e_{\tau, h}^{p}}{\mathrm{div}(\bm{\varphi}_{\tau, h}^{\bm{u}})} \mathrm{d}t \\
                    & \qquad = \int_{I_{n}} \scalar{\bm{T}_{\text{I}}^{n} - \bm{T}_{\text{II}}^{n} - \rho_{f} \partial_{t} \bm{\eta}^{\bm{w}}}{\bm{\varphi}_{\tau, h}^{\bm{u}}}
                    - \scalar{\bm{\mathcal{A}}_{h} \bm{T}_{\text{III}}^{n}}{\bm{\varphi}_{\tau, h}^{\bm{u}}}
                    + \scalar{\bm{\mathcal{A}}_{h} \bm{T}_{\text{IV}}^{n}}{\bm{\varphi}_{\tau, h}^{\bm{u}}}
                    + \scalar{\alpha \eta^{p}}{\mathrm{div}(\bm{\varphi}_{\tau, h}^{\bm{u}})} \mathrm{d}t,
                \end{split} \\
                \begin{split}
                    \label{eq:dg_var_eqn-2}
                    & \int_{I_{n}} \scalar{\partial_{t} \bm{e}_{\tau, h}^{\bm{u}} - \bm{e}_{\tau, h}^{\bm{v}}}{\bm{\varphi}_{\tau, h}^{\bm{v}}} \mathrm{d}t = 0,
                \end{split} \\
                \begin{split}
                    \label{eq:dg_var_eqn-3}
                    & \int_{I_{n}} \scalar{\rho_{f} \partial_{t} \bm{e}_{\tau, h}^{\bm{v}} + \rho_{w} \partial_{t} \bm{e}_{\tau, h}^{\bm{w}} + \bm{K}^{-1} \bm{e}_{\tau, h}^{\bm{w}}}{\bm{\varphi}_{\tau, h}^{\bm{w}}}
                    - \scalar{e_{\tau, h}^{p}}{\mathrm{div}(\bm{\varphi}_{\tau, h}^{\bm{w}})} \mathrm{d}t \\
                    & \qquad = - \int_{I_{n}} \scalar{\rho_{f} \partial_{t} \bm{\eta}^{\bm{v}} + \rho_{w} \partial_{t} \bm{\eta}^{\bm{w}} + \bm{K}^{-1} \bm{\eta}^{\bm{w}}}{\bm{\varphi}_{\tau, h}^{\bm{w}}}
                    - \scalar{\eta^{p}}{\mathrm{div}(\bm{\varphi}_{\tau, h}^{\bm{w}})} \mathrm{d}t,
                \end{split} \\
                \begin{split}
                    \label{eq:dg_var_eqn-4}
                    & \int_{I_{n}} \scalar{s_{0} \partial_{t} e_{\tau, h}^{p} + \alpha \mathrm{div}(\partial_{t} \bm{e}_{\tau, h}^{\bm{u}}) + \mathrm{div}(\bm{e}_{\tau, h}^{\bm{w}})}{\varphi_{\tau, h}^{p}} \mathrm{d}t \\
                    & \qquad = \int_{I_{n}} \scalar{T_{\text{V}}^{n}
                        - s_{0} \partial_{t} \eta^{p} - \alpha \mathrm{div}(\partial_{t} \bm{\eta}^{\bm{u}}) - \mathrm{div}(\bm{\eta}^{\bm{w}})}{\varphi_{\tau, h}^{p}} \mathrm{d}t,
                \end{split}
            \end{align}
        \end{subequations}
        for all $ (\bm{\varphi}_{\tau, h}^{\bm{u}}, \bm{\varphi}_{\tau, h}^{\bm{v}}, \bm{\varphi}_{\tau, h}^{\bm{w}}, \varphi_{\tau, h}^{p}) \in \mathbb{P}_{k-1}(I_{n}; \bm{X}_{h}) $ and $ n = 1, \ldots, N $.
    \end{lemma}
    
    By the same arguments as in~\cite[Lem.~4.3]{Bause2024Conv}, the next three lemmas provide estimates for terms on the right-hand side of \eqref{eq:dg_var_eqn}.
    
    \begin{lemma}[Estimates for \eqref{eq:dg_var_eqn-1}]
        \label{lem:dg_estimates1}
        For all $ \bm{\varphi}_{\tau, h}^{\bm{u}} \in \mathbb{P}_{k-1}(I_{n}; \bm{U}_{h}) $, the estimates
        \begin{align*}
            \abs{\int_{I_{n}} \scalar{\bm{T}_{\text{II}}^{n}}{\bm{\varphi}_{\tau, h}^{\bm{u}}} \mathrm{d}t}
            & \leq c \left( \tau_{n}^{k+1} \norm{\partial_{t}^{k+3} \bm{u}}_{L^{2}(I_{n}; \bm{L}^{2})}
            + h^{\ell+1} \norm{\partial_{t}^{2} \bm{u}}_{L^{2}(I_{n}; \bm{H}^{\ell+2})} \right) \norm{\bm{\varphi}_{\tau, h}^{\bm{u}}}_{L^{2}(I_{n}; \bm{L}^{2})}, \\
            \norm{\bm{\mathcal{A}}_{h} \bm{T}_{\text{III}}^{n}}_{L^{2}(I_{n}; \bm{L}^{2})}
            & \leq c \tau_{n}^{k+1} \norm{\bm{\mathcal{A}}_{h} \partial_{t}^{k+1} \bm{u}}_{L^{2}(I_{n}; \bm{L}^{2})}, \\
            \norm{\bm{\mathcal{A}}_{h} \bm{T}_{\text{IV}}^{n}}_{L^{2}(I_{n}; \bm{L}^{2})}
            & \leq c \tau_{n}^{k+1} \norm{\bm{\mathcal{A}}_{h} \partial_{t}^{k+1} \bm{u}}_{L^{2}(I_{n}; \bm{L}^{2})}, \\
            \abs{\int_{I_{n}} \scalar{\partial_{t} \bm{\eta}^{\bm{w}}}{\bm{\varphi}_{\tau, h}^{\bm{u}}} \mathrm{d}t}
            & \leq c \left( \tau_{n}^{k+1} \norm{\partial_{t}^{k+2} \bm{w}}_{L^{2}(I_{n}; \bm{L}^{2})}
            + h^{\ell+1} \norm{\partial_{t} \bm{w}}_{L^{2}(I_{n}; \bm{H}^{\ell+1})} \right) \norm{\bm{\varphi}_{\tau, h}^{\bm{u}}}_{L^{2}(I_{n}; \bm{L}^{2})},
        \end{align*}
        hold true.
    \end{lemma}
    
    \begin{lemma}[Estimates for \eqref{eq:dg_var_eqn-3}]
        \label{lem:dg_estimates2}
        For all $ \bm{\varphi}_{\tau, h}^{\bm{w}} \in \mathbb{P}_{k-1}(I_{n}; \bm{W}_{h}) $, the following estimates are valid:
        \begin{align*}
            \abs{\int_{I_{n}} \scalar{\partial_{t} \bm{\eta}^{\bm{v}}}{\bm{\varphi}_{\tau, h}^{\bm{w}}} \mathrm{d}t}
            & \leq c \left( \tau_{n}^{k+1} \norm{\partial_{t}^{k+3} \bm{u}}_{L^{2}(I_{n}; \bm{L}^{2})}
            + h^{\ell+1} \norm{\partial_{t}^{2} \bm{u}}_{L^{2}(I_{n}; \bm{H}^{\ell+2})} \right) \norm{\bm{\varphi}_{\tau, h}^{\bm{w}}}_{L^{2}(I_{n}; \bm{L}^{2})}, \\
            \abs{\int_{I_{n}} \scalar{\partial_{t} \bm{\eta}^{\bm{w}}}{\bm{\varphi}_{\tau, h}^{\bm{w}}} \mathrm{d}t}
            & \leq c \left( \tau_{n}^{k+1} \norm{\partial_{t}^{k+2} \bm{w}}_{L^{2}(I_{n}; \bm{L}^{2})}
            + h^{\ell+1} \norm{\partial_{t} \bm{w}}_{L^{2}(I_{n}; \bm{H}^{\ell+1})} \right) \norm{\bm{\varphi}_{\tau, h}^{\bm{w}}}_{L^{2}(I_{n}; \bm{L}^{2})}, \\
            \abs{\int_{I_{n}} \scalar{\bm{K}^{-1} \bm{\eta}^{\bm{w}}}{\bm{\varphi}_{\tau, h}^{\bm{w}}} \mathrm{d}t}
            & \leq c \left( \tau_{n}^{k+1} \norm{\partial_{t}^{k+1} \bm{w}}_{L^{2}(I_{n}; \bm{L}^{2})}
            + h^{\ell+1} \norm{\bm{w}}_{L^{2}(I_{n}; \bm{H}^{\ell+1})} \right. \\
            & \quad \left. + \tau_{n} h^{\ell+1} \norm{\partial_{t} \bm{w}}_{L^{2}(I_{n}; \bm{H}^{\ell+1})} \right) \norm{\bm{\varphi}_{\tau, h}^{\bm{w}}}_{L^{2}(I_{n}; \bm{L}^{2})}.
        \end{align*}
    \end{lemma}
    
    \begin{lemma}[Estimates for \eqref{eq:dg_var_eqn-4}]
        \label{lem:dg_estimates3}
        For all $ \varphi_{\tau, h}^{p} \in \mathbb{P}_{k-1}(I_{n}; P_{h}) $, there holds
        \begin{align*}
            \abs{\int_{I_{n}} \scalar{\partial_{t} \eta^{p}}{\varphi_{\tau, h}^{p}} \mathrm{d}t}
            & \leq c \tau_{n}^{k+1} \norm{\partial_{t}^{k+2} p}_{L^{2}(I_{n}; L^{2})} \norm{\varphi_{\tau, h}^{p}}_{L^{2}(I_{n}; L^{2})}, \\
            \abs{\int_{I_{n}} \scalar{\mathrm{div}(\partial_{t} \bm{\eta}^{\bm{u}})}{\varphi_{\tau, h}^{p}} \mathrm{d}t}
            & \leq c \left( \tau_{n}^{k+1} \norm{\partial_{t}^{k+2} \bm{u}}_{L^{2}(I_{n}; \bm{U}_{h})}
            + h^{\ell+1} \norm{\partial_{t} \bm{u}}_{L^{2}(I_{n}; \bm{H}^{\ell+2})} \right) \norm{\varphi_{\tau, h}^{p}}_{L^{2}(I_{n}; L^{2})}, \\
            \abs{\int_{I_{n}} \scalar{\mathrm{div}(\bm{\eta}^{\bm{w}})}{\varphi_{\tau, h}^{p}} \mathrm{d}t}
            & \leq c \tau_{n}^{k+1} \norm{\partial_{t}^{k+1} \bm{w}}_{L^{2}(I_{n}; \bm{H}^{1})} \norm{\varphi_{\tau, h}^{p}}_{L^{2}(I_{n}; L^{2})}.
        \end{align*}
    \end{lemma}
    
    With a proper choice of test functions in \eqref{eq:dg_var_eqn}, the above lemmas allow for an estimation of the energy norm of the discrete part of the discretization error $ \bm{e}_{\tau, h}^{\bm{x}} $, cf.~\cite{Karakashian2005Conv, Bause2024Conv}.
    
    \begin{lemma}
        \label{lem:dg_estimate_tf1}
        Let
        \begin{align*}
            \delta_{n} & := \scalar{\alpha \eta^{p}(t_{n})}{\mathrm{div}(\bm{e}_{\tau, h}^{\bm{u}}(t_{n}))}, &
            \delta_{n-1}^{+} & := \scalar{\alpha \eta^{p}(t_{n-1}^{+})}{\mathrm{div}(\bm{e}_{\tau, h}^{\bm{u}}(t_{n-1}^{+}))},
        \end{align*}
        and
        \begin{align*}
            \mathcal{E}_{t}^{n, 1} & :=
            \norm{\partial_{t}^{k+1} \bm{f}}_{L^{2}(I_{n}; \bm{L}^{2})}^{2}
            + \norm{\partial_{t}^{k+1} g}_{L^{2}(I_{n}; L^{2})}^{2}
            + \norm{\partial_{t}^{k+1} \bm{u}}_{L^{2}(I_{n}; \bm{H}^{2})}^{2}
            + \norm{\partial_{t}^{k+2} \bm{u}}_{L^{2}(I_{n}; \bm{U}_{h})}^{2}
            + \norm{\partial_{t}^{k+3} \bm{u}}_{L^{2}(I_{n}; \bm{L}^{2})}^{2} \\
            & \quad + \norm{\partial_{t}^{k+1} \bm{w}}_{L^{2}(I_{n}; \bm{H}^{1})}^{2}
            + \norm{\partial_{t}^{k+2} \bm{w}}_{L^{2}(I_{n}; \bm{L}^{2})}^{2}
            + \norm{\partial_{t}^{k+1} p}_{L^{2}(I_{n}; H^{1})}^{2}
            + \norm{\partial_{t}^{k+2} p}_{L^{2}(I_{n}; L^{2})}^{2}, \\
            \mathcal{E}_{\mathbf{x}}^{n, 1} & :=
            \norm{\partial_{t} \bm{u}}_{L^{2}(I_{n}; \bm{H}^{\ell+2})}^{2}
            + \norm{\partial_{t}^{2} \bm{u}}_{L^{2}(I_{n}; \bm{H}^{\ell+2})}^{2}
            + \norm{\bm{w}}_{L^{2}(I_{n}; \bm{H}^{\ell+1})}^{2}
            + \norm{\partial_{t} \bm{w}}_{L^{2}(I_{n}; \bm{H}^{\ell+1})}^{2}
            + \norm{\partial_{t} p}_{L^{2}(I_{n}; H^{\ell+1})}^{2}.
        \end{align*}
       Then,
        \begin{align}
            \label{eq:dg_estimate_tf1}
            \opnorm{\bm{e}_{\tau, h}^{\bm{x}}(t_{n})}_{e}^{2}
            & \leq \opnorm{\bm{e}_{\tau, h}^{\bm{x}}(t_{n-1}^{+})}_{e}^{2}
            + c \opnorm{\bm{e}_{\tau, h}^{\bm{x}}}_{L^{2}(I_{n}; \bm{L}^{2})}^{2}
            + \delta_{n}
            - \delta_{n-1}^{+}
            + c \tau_{n}^{2(k+1)} \mathcal{E}_{t}^{n, 1}
            + c h^{2(\ell+1)} \mathcal{E}_{\mathbf{x}}^{n, 1}
        \end{align}
        holds for $n = 1, \ldots, N$.
    \end{lemma}
    
    \begin{proof}
        We use the test functions
        \begin{align}
            \label{eq:prf_dg_estimate_tf1-1}
            \bm{\varphi}_{\tau, h}^{\bm{x}} :=
            \begin{pmatrix}
                \bm{\varphi}_{\tau, h}^{\bm{u}} \\
                \bm{\varphi}_{\tau, h}^{\bm{v}} \\
                \bm{\varphi}_{\tau, h}^{\bm{w}} \\
                \varphi_{\tau, h}^{p}
            \end{pmatrix}
            =
            \begin{pmatrix}
                \bm{\Pi}_{\tau}^{k-1} \bm{e}_{\tau, h}^{\bm{v}} \\
                \bm{\Pi}_{\tau}^{k-1} \bm{\mathcal{A}}_{h} \bm{e}_{\tau, h}^{\bm{u}} \\
                \bm{\Pi}_{\tau}^{k-1} \bm{e}_{\tau, h}^{\bm{w}} \\
                \Pi_{\tau}^{k-1} e_{\tau, h}^{p}
            \end{pmatrix}
        \end{align}
        in \eqref{eq:dg_var_eqn}.
        By exploiting Lemma~\ref{lem:dg_z_property_pointwise} and \eqref{eq:dg_var_eqn-2}, it follows
        \begin{align}
            \label{eq:prf_dg_estimate_tf1-2}
            \partial_{t} \bm{e}_{\tau, h}^{\bm{u}}(t_{n, i}^{\text{G}}) = \bm{e}_{\tau, h}^{\bm{v}}(t_{n, i}^{\text{G}})
            \quad i = 1, \ldots, k, \, n = 1, \ldots, N.
        \end{align}
        Due to~\eqref{eq:local_L2_projection} we have
        \begin{align*}
            \int_{I_{n}} \scalar{\bm{K}^{-1} \bm{e}_{\tau, h}^{\bm{w}}}{\bm{\Pi}_{\tau}^{k-1} \bm{e}_{\tau, h}^{\bm{w}}} \mathrm{d}t
            = \int_{I_{n}} \scalar{\bm{K}^{-1} \bm{\Pi}_{\tau}^{k-1} \bm{e}_{\tau, h}^{\bm{w}}}{\bm{\Pi}_{\tau}^{k-1} \bm{e}_{\tau, h}^{\bm{w}}} \mathrm{d}t
            \geq 0.
        \end{align*}
        The exactness of~\eqref{eq:quadrature_G} for polynomials in $ \mathbb{P}_{2k-1}(I_{n}; \mathbb{R}) $ together with Lemma~\ref{lem:local_L2_projection} show
        \begin{align*}
            \int_{I_{n}} \scalar{s_{0} \partial_{t} e_{\tau, h}^{p}}{\Pi_{\tau}^{k-1} e_{\tau, h}^{p}} \mathrm{d}t
            = \int_{I_{n}} \scalar{s_{0} \partial_{t} e_{\tau, h}^{p}}{e_{\tau, h}^{p}} \mathrm{d}t
            = \frac{s_{0}}{2} \norm{e_{\tau, h}^{p}(t_{n})}_{L^{2}(\Omega)}^{2}
            - \frac{s_{0}}{2} \norm{e_{\tau, h}^{p}(t_{n-1}^{+})}_{L^{2}(\Omega)}^{2}.
        \end{align*}
        Similarly, we find
        \begin{align*}
            & \quad \int_{I_{n}} \scalar{\bar{\rho} \partial_{t} \bm{e}_{\tau, h}^{\bm{v}} + \rho_{f} \partial_{t} \bm{e}_{\tau, h}^{\bm{w}}}{\bm{\Pi}_{\tau}^{k-1} \bm{e}_{\tau, h}^{\bm{v}}}
            + \scalar{\rho_{f} \partial_{t} \bm{e}_{\tau, h}^{\bm{v}} + \rho_{w} \partial_{t} \bm{e}_{\tau, h}^{\bm{w}}}{\bm{\Pi}_{\tau}^{k-1} \bm{e}_{\tau, h}^{\bm{w}}} \mathrm{d}t \\
            & = \int_{I_{n}} \scalar{\bm{M}_{\rho}
                \begin{pmatrix}
                    \partial_{t} \bm{e}_{\tau, h}^{\bm{v}} \\
                    \partial_{t} \bm{e}_{\tau, h}^{\bm{w}}
                \end{pmatrix}
            }{
                \begin{pmatrix}
                    \bm{e}_{\tau, h}^{\bm{v}} \\
                    \bm{e}_{\tau, h}^{\bm{w}}
                \end{pmatrix}
            } \mathrm{d}t
            = \frac{1}{2} \norm{\bm{M}_{\rho}^{1/2}
                \begin{pmatrix}
                    \bm{e}_{\tau, h}^{\bm{v}}(t_{n}) \\
                    \bm{e}_{\tau, h}^{\bm{w}}(t_{n})
                \end{pmatrix}
            }_{\bm{L}^{2}(\Omega)}^{2}
            - \frac{1}{2} \norm{\bm{M}_{\rho}^{1/2}
                \begin{pmatrix}
                    \bm{e}_{\tau, h}^{\bm{v}}(t_{n-1}^{+}) \\
                    \bm{e}_{\tau, h}^{\bm{w}}(t_{n-1}^{+})
                \end{pmatrix}
            }_{\bm{L}^{2}(\Omega)}^{2}
        \end{align*}
        and
        \begin{align*}
            \int_{I_{n}} \scalar{\partial_{t} \bm{e}_{\tau, h}^{\bm{u}}}{\bm{\Pi}_{\tau}^{k-1} \bm{\mathcal{A}}_{h} \bm{e}_{\tau, h}^{\bm{u}}} \mathrm{d}t
            = \int_{I_{n}} \scalar{\partial_{t} \bm{e}_{\tau, h}^{\bm{u}}}{\bm{\mathcal{A}}_{h} \bm{e}_{\tau, h}^{\bm{u}}} \mathrm{d}t
            = \frac{1}{2} \, a_{h}(\bm{e}_{\tau, h}^{\bm{u}}(t_{n}), \bm{e}_{\tau, h}^{\bm{u}}(t_{n}))
            - \frac{1}{2} \, a_{h}(\bm{e}_{\tau, h}^{\bm{u}}(t_{n-1}^{+}), \bm{e}_{\tau, h}^{\bm{u}}(t_{n-1}^{+})).
        \end{align*}
        With~\eqref{eq:prf_dg_estimate_tf1-1} and~\eqref{eq:prf_dg_estimate_tf1-2} one obtains
        \begin{align*}
            \int_{I_{n}} \scalar{\bm{\mathcal{A}}_{h} \bm{e}_{\tau, h}^{\bm{u}}}{\bm{\Pi}_{\tau}^{k-1} \bm{e}_{\tau, h}^{\bm{v}}}
            - \scalar{\bm{e}_{\tau, h}^{\bm{v}}}{\bm{\Pi}_{\tau}^{k-1} \bm{\mathcal{A}}_{h} \bm{e}_{\tau, h}^{\bm{u}}} \mathrm{d}t & = 0, \\
            \int_{I_{n}} \scalar{\alpha \mathrm{div}(\partial_{t} \bm{e}_{\tau, h}^{\bm{u}})}{\Pi_{\tau}^{k-1} e_{\tau, h}^{p}}
            - \scalar{\alpha e_{\tau, h}^{p}}{\mathrm{div}(\bm{\Pi}_{\tau}^{k-1} \bm{e}_{\tau, h}^{\bm{v}})} \mathrm{d}t & = 0, \\
            \int_{I_{n}} \scalar{\mathrm{div}(\bm{e}_{\tau, h}^{\bm{w}})}{\Pi_{\tau}^{k-1} e_{\tau, h}^{p}}
            - \scalar{e_{\tau, h}^{p}}{\mathrm{div}(\bm{\Pi}_{\tau}^{k-1} \bm{e}_{\tau, h}^{\bm{w}})} \mathrm{d}t & = 0.
        \end{align*}
        Next, we use~\cite[Eq.~(4.36)]{Bause2024Conv} and integrate by parts
        \begin{align*}
            & \quad \int_{I_{n}} \scalar{\eta^{p}}{\mathrm{div}(\bm{\Pi}_{\tau}^{k-1} \bm{e}_{\tau, h}^{\bm{v}})} \mathrm{d}t
            = \int_{I_{n}} \scalar{\eta^{p}}{\mathrm{div}(\partial_{t} \bm{e}_{\tau, h}^{\bm{u}})} \mathrm{d}t \\
            & = - \int_{I_{n}} \scalar{\partial_{t} \eta^{p}}{\mathrm{div}(\bm{e}_{\tau, h}^{\bm{u}})} \mathrm{d}t
            + \scalar{\eta^{p}(t_{n})}{\mathrm{div}(\bm{e}_{\tau, h}^{\bm{u}}(t_{n}))}
            - \scalar{\eta^{p}(t_{n-1}^{+})}{\mathrm{div}(\bm{e}_{\tau, h}^{\bm{u}}(t_{n-1}^{+}))},
        \end{align*}
        where
        \begin{align*}
            \abs{\int_{I_{n}} \scalar{\partial_{t} \eta^{p}}{\mathrm{div}(\bm{e}_{\tau, h}^{\bm{u}})} \mathrm{d}t}
            & \leq c \left( \tau_{n}^{k+1} \norm{\partial_{t}^{k+2} p}_{L^{2}(I_{n}; L^{2})}
            + h^{\ell+1} \norm{\partial_{t} p}_{L^{2}(I_{n}; H^{\ell+1})} \right)
            \norm{\bm{e}_{\tau, h}^{\bm{u}}}_{L^{2}(I_{n}; \bm{U}_{h})}.
        \end{align*}
        Further, in view of $ \mathrm{div}(\bm{W}_{h}) = P_{h} $ and \eqref{eq:dg_projection_Ph}, we have
        \begin{align*}
            \int_{I_{n}} \scalar{\eta^{p}}{\mathrm{div}(\bm{\Pi}_{\tau}^{k-1} \bm{e}_{\tau, h}^{\bm{w}})} \mathrm{d}t
            & = \int_{I_{n}} \scalar{p - \mathcal{I}_{\tau}^{\text{GL}} p}{\mathrm{div}(\bm{\Pi}_{\tau}^{k-1} \bm{e}_{\tau, h}^{\bm{w}})} \mathrm{d}t
            = - \int_{I_{n}} \scalar{\nabla (p - \mathcal{I}_{\tau}^{\text{GL}} p)}{\bm{\Pi}_{\tau}^{k-1} \bm{e}_{\tau, h}^{\bm{w}}} \mathrm{d}t.
        \end{align*}
        Hence, estimate~\eqref{eq:interpolation_error} yields
        \begin{align*}
            \abs{\int_{I_{n}} \scalar{\nabla (p - \mathcal{I}_{\tau}^{\text{GL}} p)}{\bm{\Pi}_{\tau}^{k-1} \bm{e}_{\tau, h}^{\bm{w}}} \mathrm{d}t}
            & \leq c \tau_{n}^{k+1} \norm{\partial_{t}^{k+1} p}_{L^{2}(I_{n}; H^{1})} \norm{\bm{\Pi}_{\tau}^{k-1} \bm{e}_{\tau, h}^{\bm{w}}}_{L^{2}(I_{n}; \bm{L}^{2})}.
        \end{align*}
        Finally, with~\eqref{eq:interpolation_error}, Lemma~\ref{lem:dg_estimates1}--Lemma~\ref{lem:dg_estimates3} and Young's inequality
        we conclude assertion~\eqref{eq:dg_estimate_tf1}.
    \end{proof}
    
    Subsequently, we estimate the first term on the right-hand side of \eqref{eq:dg_estimate_tf1}.
    
    \begin{lemma}
        \label{lem:dg_estimate_limits}
        There holds
        \begin{align}
            \label{eq:dg_estimate_limits}
            \opnorm{\bm{e}_{\tau, h}^{\bm{x}}(t_{n-1}^{+})}_{e}^{2}
            & \leq (1 + \tau_{n-1}) \opnorm{\bm{e}_{\tau, h}^{\bm{x}}(t_{n-1})}_{e}^{2}
            + c (1 + \tau_{n-1}) \tau_{n-1}^{2(k+1)} \norm{\partial_{t}^{k+2} \bm{u}}_{L^{2}(I_{n-1}; \bm{U}_{h})}^{2}
        \end{align}
        for $ n = 2, \ldots, N $.
    \end{lemma}
    
    \begin{proof}
        Note that $ (\bm{u}_{\tau, h}, \bm{v}_{\tau, h}, \bm{w}_{\tau, h}, p_{\tau, h}) $ as well as $ (\bm{z}_{2}, \bm{z}_{3}, z_{4}) $ are continuous in time.
        Thus, the equation
        \begin{align*}
            \opnorm{\bm{e}_{\tau, h}^{\bm{x}}(t_{n-1}^{+})}_{e}^{2}
            & = \frac{1}{2} \, a_{h}(\bm{e}_{\tau, h}^{\bm{u}}(t_{n-1}^{+}), \bm{e}_{\tau, h}^{\bm{u}}(t_{n-1}^{+}))
            + \frac{1}{2} \norm{\bm{M}_{\rho}^{1/2}
                \begin{pmatrix}
                    \bm{e}_{\tau, h}^{\bm{v}}(t_{n-1}) \\
                    \bm{e}_{\tau, h}^{\bm{w}}(t_{n-1})
                \end{pmatrix}
            }_{\bm{L}^{2}(\Omega)}^{2}
            + \frac{s_{0}}{2} \norm{e_{\tau, h}^{p}(t_{n-1})}_{L^{2}(\Omega)}^{2}
        \end{align*}
        is satisfied.
        Employing the decomposition
        \begin{align*}
            \bm{e}_{\tau, h}^{\bm{u}}(t_{n-1}^{+})
            = \left[ \bm{z}_{1}(t_{n-1}^{+}) - \bm{z}_{1}(t_{n-1}) \right]
            + \bm{e}_{\tau, h}^{\bm{u}}(t_{n-1}),
        \end{align*}
        we find the estimate
        \begin{align*}
            a_{h}(\bm{e}_{\tau, h}^{\bm{u}}(t_{n-1}^{+}), \bm{e}_{\tau, h}^{\bm{u}}(t_{n-1}^{+}))
            & = a_{h}(\bm{e}_{\tau, h}^{\bm{u}}(t_{n-1}), \bm{e}_{\tau, h}^{\bm{u}}(t_{n-1}))
            + 2 \, a_{h}(\bm{e}_{\tau, h}^{\bm{u}}(t_{n-1}), \bm{z}_{1}(t_{n-1}^{+}) - \bm{z}_{1}(t_{n-1})) \\
            & \quad + a_{h}(\bm{z}_{1}(t_{n-1}^{+}) - \bm{z}_{1}(t_{n-1}), \bm{z}_{1}(t_{n-1}^{+}) - \bm{z}_{1}(t_{n-1})) \\
            & \leq [1 + \tau_{n-1}] \, a_{h}(\bm{e}_{\tau, h}^{\bm{u}}(t_{n-1}), \bm{e}_{\tau, h}^{\bm{u}}(t_{n-1})) \\
            & \quad + [1 + \tau_{n-1}^{-1}] \, a_{h}(\bm{z}_{1}(t_{n-1}^{+}) - \bm{z}_{1}(t_{n-1}), \bm{z}_{1}(t_{n-1}^{+}) - \bm{z}_{1}(t_{n-1})).
        \end{align*}
        Further, as $ a_{h}(\cdot, \cdot) $ is bounded, we get
        \begin{align*}
            & \quad a_{h}(\bm{z}_{1}(t_{n-1}^{+}) - \bm{z}_{1}(t_{n-1}), \bm{z}_{1}(t_{n-1}^{+}) - \bm{z}_{1}(t_{n-1})) \\
            & \leq c \left( \norm{\bm{z}_{1}(t_{n-1}^{+}) - \bm{\mathcal{R}}_{h} \bm{u}(t_{n-1})}_{\bm{U}_{h}}
            + \norm{\bm{\mathcal{R}}_{h} \bm{u}(t_{n-1}) - \bm{z}_{1}(t_{n-1})}_{\bm{U}_{h}} \right)^{2} \\
            & \leq c \norm{\int_{t_{n-2}}^{t_{n-1}} \bm{\mathcal{R}}_{h} \partial_{t} \bm{u} -  \bm{\mathcal{I}}_{\tau}^{\text{GL}} \bm{\mathcal{R}}_{h} \partial_{t} \bm{u} \mathrm{d}s}_{\bm{U}_{h}}^{2}
            \leq c \tau_{n-1} \tau_{n-1}^{2(k+1)} \norm{\partial_{t}^{k+2} \bm{u}}_{L^{2}(I_{n-1}; \bm{U}_{h})}^{2}.
        \end{align*}
        Together we obtain \eqref{eq:dg_estimate_limits}.
    \end{proof}
    
    The following lemma provides a bound on the second term on the right-hand side of~\eqref{eq:dg_estimate_tf1}.
    
    \begin{lemma}
        \label{lem:dg_estimate_tf2}
        Let $ \mathcal{E}_{t}^{n, 1}, \mathcal{E}_{\mathbf{x}}^{n, 1} $ be defined as in Lemma~\ref{lem:dg_estimate_tf1} and
        \begin{align*}
            \mathcal{E}_{t}^{n, 2} & := \norm{\partial_{t}^{k+1} p}_{L^{\infty}(I_{n}; L^{2})}^{2}, &
            \mathcal{E}_{\mathbf{x}}^{n, 2} & := \norm{p}_{L^{\infty}(I_{n}; H^{\ell+1})}^{2}
            + \tau_{n}^{2} \norm{\partial_{t} p}_{L^{\infty}(I_{n}; H^{\ell+1})}^{2}.
        \end{align*}
        Then, there holds
        \begin{align*}
            \opnorm{\bm{e}_{\tau, h}^{\bm{x}}}_{L^{2}(I_{n}; \bm{L}^{2})}^{2}
            & \leq c \tau_{n} \left( \opnorm{\bm{e}_{\tau, h}^{\bm{x}}(t_{n-1}^{+})}_{e}^{2}
            + \tau_{n}^{2(k+1)} \left( \mathcal{E}_{t}^{n, 1}
            + \mathcal{E}_{t}^{n, 2} \right)
            + h^{2(\ell+1)} \left( \mathcal{E}_{\mathbf{x}}^{n, 1}
            + \mathcal{E}_{\mathbf{x}}^{n, 2} \right)
            \right)
        \end{align*}
        for $ n = 1, \ldots, N $.
    \end{lemma}
    
    \begin{proof}
        Using the coefficients $ \bm{e}_{n, j}^{\bm{x}} := (\bm{e}_{n, j}^{\bm{u}}, \bm{e}_{n, j}^{\bm{v}}, \bm{e}_{n, j}^{\bm{w}}, e_{n, j}^{p}) \in \bm{X}_{h} $ for $ j = 0, \ldots, k $, we can represent $ \bm{e}_{\tau, h}^{\bm{x}} $ as
        \begin{align}
            \label{eq:prf_dg_estimate_tf2-1}
            \bm{e}_{\tau, h}^{\bm{x}}(t) & = \sum_{j=0}^{k} \bm{e}_{n, j}^{\bm{x}} L_{n, j}^{\text{G}, 0}(t).
        \end{align}
        Further, we define the test functions
        \begin{align}
            \label{eq:prf_dg_estimate_tf2-2}
            \begin{aligned}
                \bm{\varphi}_{\tau, h}^{\bm{u}}(t) & = \sum_{i=1}^{k} (\hat{t}_{i}^{\text{G}})^{-1/2} \bm{\tilde{e}}_{n, i}^{\bm{v}} L_{n, i}^{\text{G}}(t), & \qquad
                \bm{\varphi}_{\tau, h}^{\bm{v}}(t) & = \sum_{i=1}^{k} (\hat{t}_{i}^{\text{G}})^{-1/2} \bm{\mathcal{A}}_{h} \bm{\tilde{e}}_{n, i}^{\bm{u}} L_{n, i}^{\text{G}}(t), \\
                \bm{\varphi}_{\tau, h}^{\bm{w}}(t) & = \sum_{i=1}^{k} (\hat{t}_{i}^{\text{G}})^{-1/2} \bm{\tilde{e}}_{n, i}^{\bm{w}} L_{n, i}^{\text{G}}(t), &
                \varphi_{\tau, h}^{p}(t) & = \sum_{i=1}^{k} (\hat{t}_{i}^{\text{G}})^{-1/2} \tilde{e}_{n, i}^{p} L_{n, i}^{\text{G}}(t),
            \end{aligned}
        \end{align}
        for \eqref{eq:dg_var_eqn}, where $ \bm{\tilde{e}}_{n, i}^{\bm{x}} := (\bm{\tilde{e}}_{n, i}^{\bm{u}}, \bm{\tilde{e}}_{n, i}^{\bm{v}}, \bm{\tilde{e}}_{n, i}^{\bm{w}}, \tilde{e}_{n, i}^{p}) := (\hat{t}_{i}^{\text{G}})^{-1/2} \bm{e}_{n, i}^{\bm{x}} $.
        Then, exploiting \eqref{eq:quadrature_G}, from \eqref{eq:prf_dg_estimate_tf2-1}--\eqref{eq:prf_dg_estimate_tf2-2} we find
        \begin{align*}
            \int_{I_{n}} \scalar{\bm{\mathcal{A}}_{h} \bm{e}_{\tau, h}^{\bm{u}}}{\bm{\varphi}_{\tau, h}^{\bm{u}}}
            - \scalar{\bm{e}_{\tau, h}^{\bm{v}}}{\bm{\varphi}_{\tau, h}^{\bm{v}}} \mathrm{d}t
            = \tau_{n} \sum_{j=1}^{k} \hat{\omega}_{j}^{\text{G}} (\hat{t}_{j}^{\text{G}})^{-1} \left[ \scalar{\bm{\mathcal{A}}_{h} \bm{e}_{n, j}^{\bm{u}}}{\bm{e}_{n, j}^{\bm{v}}}
            - \scalar{\bm{e}_{n, j}^{\bm{v}}}{\bm{\mathcal{A}}_{h} \bm{e}_{n, j}^{\bm{u}}} \right]
            = 0.
        \end{align*}
        Moreover, with \eqref{eq:prf_dg_estimate_tf1-2}, we get
        \begin{align*}
            \int_{I_{n}} \scalar{\alpha \mathrm{div}(\partial_{t} \bm{e}_{\tau, h}^{\bm{u}})}{\varphi_{\tau, h}^{p}}
            - \scalar{\alpha e_{\tau, h}^{p}}{\mathrm{div}(\bm{\varphi}_{\tau, h}^{\bm{u}})} \mathrm{d}t
            & = 0, \\
            \int_{I_{n}} \scalar{\mathrm{div}(\bm{e}_{\tau, h}^{\bm{w}})}{\varphi_{\tau, h}^{p}}
            - \scalar{e_{\tau, h}^{p}}{\mathrm{div}(\bm{\varphi}_{\tau, h}^{\bm{w}})} \mathrm{d}t
            & = 0.
        \end{align*}
        Additionally, the final estimation will take advantage of
        \begin{align*}
            \int_{I_{n}} \scalar{\bm{K}^{-1} \bm{e}_{\tau, h}^{\bm{w}}}{\bm{\varphi}_{\tau, h}^{\bm{w}}} \mathrm{d}t
            & = \tau_{n} \sum_{j=1}^{k} \hat{\omega}_{j}^{\text{G}} (\hat{t}_{j}^{\text{G}})^{-1} \scalar{\bm{K}^{-1} \bm{e}_{n, j}^{\bm{w}}}{\bm{e}_{n, j}^{\bm{w}}}
            \geq 0.
        \end{align*}
        As in \cite{Karakashian2005Conv, Bause2024Conv}, we define $ \bm{M} \in \mathbb{R}^{k \times k} $ and $ \bm{m}_{0} \in \mathbb{R}^{k} $ via
        \begin{align*}
            m_{ij} := \int_{I_{n}} \partial_{t} L_{n, j}^{\text{G}, 0} L_{n, i}^{\text{G}} \mathrm{d}t
            \qquad \text{and} \qquad
            m_{i0} := \int_{I_{n}} \partial_{t} L_{n, 0}^{\text{G}, 0} L_{n, i}^{\text{G}} \mathrm{d}t.
        \end{align*}
        According to \cite{Karakashian1999Aspa}, the matrix $ \pmb{\widetilde{M}} = (\widetilde{m}_{ij})_{ij} := \bm{D}^{-1/2} \bm{M} \bm{D}^{1/2} $ where $ \bm{D} := \mathrm{diag}(\hat{t}_{1}^{\text{G}}, \ldots, \hat{t}_{k}^{\text{G}}) $ is positive definite.
        Next, we have
        \begin{align*}
            \begin{aligned}
                \Gamma_{1} & := \int_{I_{n}} \scalar{\bm{M}_{\rho}
                    \begin{pmatrix}
                        \partial_{t} \bm{e}_{\tau, h}^{\bm{v}} \\
                        \partial_{t} \bm{e}_{\tau, h}^{\bm{w}}
                    \end{pmatrix}
                }{
                    \begin{pmatrix}
                        \bm{\varphi}_{\tau, h}^{\bm{u}} \\
                        \bm{\varphi}_{\tau, h}^{\bm{w}}
                    \end{pmatrix}
                }
                + \scalar{\partial_{t} \bm{e}_{\tau, h}^{\bm{u}}}{\bm{\varphi}_{\tau, h}^{\bm{v}}}
                + \scalar{s_{0} \partial_{t} e_{\tau, h}^{p}}{\varphi_{\tau, h}^{p}} \mathrm{d}t \\
                & = \sum_{i, j=1}^{k} \widetilde{m}_{ij} \left[
                \scalar{\bm{M}_{\rho}
                    \begin{pmatrix}
                        \bm{\tilde{e}}_{n, j}^{\bm{v}} \\
                        \bm{\tilde{e}}_{n, j}^{\bm{w}}
                    \end{pmatrix}
                }{
                    \begin{pmatrix}
                        \bm{\tilde{e}}_{n, i}^{\bm{v}} \\
                        \bm{\tilde{e}}_{n, i}^{\bm{w}}
                    \end{pmatrix}
                }
                + \scalar{\bm{\tilde{e}}_{n, j}^{\bm{u}}}{\bm{\mathcal{A}}_{h} \bm{\tilde{e}}_{n, i}^{\bm{u}}}
                + \scalar{s_{0} \tilde{e}_{n, j}^{p}}{\tilde{e}_{n, i}^{p}} \right]
                \\
                & \quad + \sum_{i=1}^{k} (\hat{t}_{i}^{\text{G}})^{-1/2} m_{i0} \left[
                \scalar{\bm{M}_{\rho}
                    \begin{pmatrix}
                        \bm{e}_{\tau, h}^{\bm{v}}(t_{n-1}^{+}) \\
                        \bm{e}_{\tau, h}^{\bm{w}}(t_{n-1}^{+})
                    \end{pmatrix}
                }{
                    \begin{pmatrix}
                        \bm{\tilde{e}}_{n, i}^{\bm{v}} \\
                        \bm{\tilde{e}}_{n, i}^{\bm{w}}
                    \end{pmatrix}
                }
                + \scalar{\bm{e}_{\tau, h}^{\bm{u}}(t_{n-1}^{+})}{\bm{\mathcal{A}}_{h} \bm{\tilde{e}}_{n, i}^{\bm{u}}}
                + \scalar{s_{0} e_{\tau, h}^{p}(t_{n-1}^{+})}{\tilde{e}_{n, i}^{p}} \right],
            \end{aligned}
        \end{align*}
        which can be estimated by
        \begin{align*}
            \Gamma_{1}
            & \geq c_{1} \sum_{i=1}^{k} \opnorm{\bm{\tilde{e}}_{n, i}^{\bm{x}}}^{2}
            - c_{2} \opnorm{\bm{e}_{\tau, h}^{\bm{x}}(t_{n-1}^{+})} \left( \sum_{i=1}^{k} \opnorm{\bm{\tilde{e}}_{n, i}^{\bm{x}}}^{2} \right)^{1/2}
            \geq c_{3} \sum_{i=1}^{k} \opnorm{\bm{\tilde{e}}_{n, i}^{\bm{x}}}^{2}
            - c_{4} \opnorm{\bm{e}_{\tau, h}^{\bm{x}}(t_{n-1}^{+})}^{2}.
        \end{align*}
        Utilizing the equivalence
        \begin{align*}
            c_{1} \tau_{n} \sum_{i=0}^{k} \opnorm{\bm{e}_{n, i}^{\bm{x}}}^{2}
            \leq \opnorm{\bm{e}_{\tau, h}^{\bm{x}}}_{L^{2}(I_{n}; \bm{L}^{2})}^{2}
            \leq c_{2} \tau_{n} \sum_{i=0}^{k} \opnorm{\bm{e}_{n, i}^{\bm{x}}}^{2},
        \end{align*}
        cf.~\cite{Karakashian2005Conv}, we infer
        \begin{align*}
            \tau_{n} \Gamma_{1}
            \geq c_{1} \opnorm{\bm{e}_{\tau, h}^{\bm{x}}}_{L^{2}(I_{n}; \bm{L}^{2})}^{2} - c_{2} \tau_{n} \opnorm{\bm{e}_{\tau, h}^{\bm{x}}(t_{n-1}^{+})}^{2}.
        \end{align*}
        The last term on the right-hand side of \eqref{eq:dg_var_eqn-1} which was not estimated in Lemma~\ref{lem:dg_estimates1} after multiplication with $ \tau_{n} $ can be bounded as follows (cf.~\cite[Lem.~4.5]{Bause2024Conv})
        \begin{align*}
            \tau_{n} \int_{I_{n}} \scalar{\alpha \eta^{p}}{\mathrm{div}(\bm{\varphi}_{\tau, h}^{\bm{u}})} \mathrm{d}t
            & \leq c \varepsilon \norm{\bm{e}_{\tau, h}^{\bm{u}}}_{L^{2}(I_{n}; \bm{U}_{h})}^{2}
            + c \tau_{n} \tau_{n}^{2(k+1)} \norm{\partial_{t}^{k+1} p}_{L^{\infty}(I_{n}; L^{2})}^{2} \\
            & \quad + c \tau_{n} h^{2(\ell+1)} \left(
            \norm{p}_{L^{\infty}(I_{n}; H^{\ell+1})}^{2}
            + \tau_{n}^{2} \norm{\partial_{t} p}_{L^{\infty}(I_{n}; H^{\ell+1})}^{2} \right),
        \end{align*}
        where we have used Lemma~\ref{lem:dg_projection_errors_z}.
        For the treatment of the other terms, we refer to Lemma~\ref{lem:dg_estimate_tf1}.
        The final result follows from application of Young's inequality.
    \end{proof}
    
    Similarly to~\cite[Lem.~4.7]{Bause2024Conv}, the next lemma can be shown.
    
    \begin{lemma}
        \label{lem:dg_delta_estimate}
        For $ n = 2, \ldots, N $, we have
        \begin{align*}
            \delta_{n} - \delta_{n-1}^{+}
            & \leq \delta_{n} - \delta_{n-1}
            + c \tau_{n-1} \left( h^{2(\ell+1)} \norm{p}_{L^{\infty}(I_{n-1}; H^{\ell+1})}^{2}
            + \tau_{n-1}^{2(k+1)} \norm{\partial_{t}^{k+2} \bm{u}}_{L^{\infty}(I_{n-1}; \bm{U}_{h})}^{2} \right)
        \end{align*}
        and
        \begin{align*}
            \abs{\delta_{1} - \delta_{0}^{+}}
            \leq c h^{2(\ell+1)} \left( \norm{p_{0}}_{H^{\ell+1}(\Omega)}^{2}
            + \norm{p(t_{1})}_{H^{\ell+1}(\Omega)}^{2}
            + \norm{\bm{u}_{0}}_{\bm{H}^{\ell+2}(\Omega)}^{2} \right)
            + \varepsilon \opnorm{\bm{e}_{\tau, h}^{\bm{x}}(t_{1})}_{e}^{2}.
        \end{align*}
    \end{lemma}
    
    Combining the estimates in this section and using the discrete Gronwall lemma, we arrive at our main result.
    
    \begin{theorem}
        \label{thm:dg_discretization_error}
        Let $ \bm{x} := (\bm{u}, \bm{v}, \bm{w}, p) $, $ \bm{x}_{\tau, h} := (\bm{u}_{\tau, h}, \bm{v}_{\tau, h}, \bm{w}_{\tau, h}, p_{\tau, h}) $ and $ \tau := \max_{n=1, \ldots, N} \tau_{n} $.
        There holds
        \begin{align*}
            \opnorm{\bm{x}(t) - \bm{x}_{\tau, h}(t)}
            & \leq c \left( \tau^{k+1} + h^{\ell+1} \right) \quad \forall t \in [0, T].
        \end{align*}
    \end{theorem}
    
    \begin{proof}
        The proof uses arguments from \cite{Karakashian2005Conv, Bause2024Conv}.
        First, using Lemmas~\ref{lem:dg_estimate_tf1}--\ref{lem:dg_delta_estimate}, we obtain
        \begin{align*}
            A_{n}
            & \leq A_{n-1}
            + (\tau_{n-1} + c \tau_{n} + c \tau_{n-1} \, \tau_{n}) \, A_{n-1}
            + \delta_{n} - \delta_{n-1}
            + c \left( \tau^{2(k+1)} + h^{2(\ell+1)} \right)
        \end{align*}
        for $ n = 2, \ldots, N $, where $ A_{n} := \opnorm{\bm{e}_{\tau, h}^{\bm{x}}(t_{n})}_{e}^{2} $.
        Then, we repeatedly employ the estimate for $ A_{m-1}, \, m = n, \ldots, 2 $ to conclude
        \begin{align*}
            A_{n}
            & \leq A_{1}
            + \sum_{m=2}^{n} (\tau_{m-1} + c \tau_{m} + c \tau_{m-1} \, \tau_{m}) \, A_{m-1}
            + \abs{\delta_{n}} + \abs{\delta_{1}}
            + c \left( \tau^{2(k+1)} + h^{2(\ell+1)} \right).
        \end{align*}
        In addition, the combination of Lemma~\ref{lem:dg_estimate_tf1}, Lemma~\ref{lem:dg_estimate_tf2}, Lemma~\ref{lem:dg_delta_estimate} and \eqref{eq:dg_assumptions} gives
        \begin{align*}
            A_{1}
            & \leq c \left( \tau^{2(k+1)} + h^{2(\ell+1)} \right)
            + \varepsilon \, A_{1}.
        \end{align*}
        Hence, we infer for sufficiently small $ \varepsilon > 0 $ that
        \begin{align*}
            A_{1} & \leq c \left( \tau^{2(k+1)} + h^{2(\ell+1)} \right).
        \end{align*}
        Further, we utilize Lemma~\ref{lem:dg_projection_errors} followed by Young's inequality to get
        \begin{align*}
            \abs{\delta_{1}} \leq c h^{2(\ell+1)} + A_{1}
            \hspace{4em} \text{and} \hspace{4em}
            \abs{\delta_{n}} \leq c h^{2(\ell+1)} + \varepsilon \, A_{n}.
        \end{align*}
        Due to the discrete Gronwall lemma, we have
        \begin{align}
            \label{eq:prf_dg_discretization_error1}
            \opnorm{\bm{e}_{\tau, h}^{\bm{x}}(t_{n})}_{e}^{2}
            \leq c \left( \tau^{2(k+1)} + h^{2(\ell+1)} \right),
            \qquad n = 0, \ldots, N.
        \end{align}
        Next, we make use of Lemmas~\ref{lem:dg_estimate_limits}--\ref{lem:dg_estimate_tf2}, \eqref{eq:prf_dg_discretization_error1} and \eqref{eq:dg_assumptions},
        which results in
        \begin{align*}
            \opnorm{\bm{e}_{\tau, h}^{\bm{x}}}_{L^{2}(I_{n}; \bm{L}^{2})}^{2}
            \leq c \tau_{n} \left( \tau^{2(k+1)} + h^{2(\ell+1)} \right)
        \end{align*}
        for $ n = 1, \ldots, N $.
        Then, using the inverse estimate \eqref{eq:Linfty_L2_inverse_inequality}, we obtain
        \begin{align}
            \label{eq:prf_dg_discretization_error2}
            \opnorm{\bm{e}_{\tau, h}^{\bm{x}}(t)}^{2}
            \leq c \left( \tau^{2(k+1)} + h^{2(\ell+1)} \right) \qquad \forall t \in [0, T].
        \end{align}
        Lastly, decomposition \eqref{eq:dg_eta_e_splitting}, Lemma~\ref{lem:dg_projection_errors_z} and \eqref{eq:prf_dg_discretization_error2} complete the proof.
    \end{proof}

    \section{Numerical experiments}
    \label{sec:dg_numerical_experiments}
    In this section, we analyze the numerical convergence for the continuous Galerkin-Petrov method in time proposed in Section~\ref{subsec:dg_discrete_formulation}.
    
    In the following, we study the discretization errors
    \begin{align*}
        \bm{e}^{\bm{u}} & := \bm{u} - \bm{u}_{\tau, h}, &
        \bm{e}^{\bm{v}} & := \bm{v} - \bm{v}_{\tau, h}, &
        \bm{e}^{\bm{w}} & := \bm{w} - \bm{w}_{\tau, h}, &
        e^{p} & := p - p_{\tau, h},
    \end{align*}
    in the $ \norm{\cdot}_{L^{\infty}(I; \bm{L}^{2})} $ norm.
    In order to compute this norm numerically, we approximate it by
    \begin{align*}
        \norm{\bm{f}}_{L^{\infty}(I; \bm{L}^{2})}
        & \approx \max_{t \in I_{\tau}} \left( \int_{\Omega} \, \norm{\bm{f}(t)}^{2} \mathrm{d}\mathbf{x} \right)^{1/2},
    \end{align*}
    where the set of discrete time points is specified by
    \begin{align*}
        I_{\tau} := \{ t_{n, i} \mid t_{n, i} = t_{n-1} + \frac{i}{100} \, \tau_{n}, \ i = 0, \ldots, 100, \, n = 1, \ldots, N \}.
    \end{align*}
    Next, we choose the right-hand sides of problem \eqref{eq:dg_dyn_biot} such that
    \begin{align*}
        \bm{u} & =
        \bm{w} =
        \begin{pmatrix}
            \sin(\pi t) y^{2} (1 - y)^{2} (4 x^{3} - 6 x^{2} + 2 x) \\
            \sin(\pi t) x^{2} (1 - x)^{2} (4 y^{3} - 6 y^{2} + 2 y)
        \end{pmatrix}, \\
        p & =
        (-\rho_{w} \pi \cos(\pi t) + (\rho_{f} \pi^{2} - 1) \sin(\pi t)) \left( x^{2} (1 - x)^{2} y^{2} (1 - y)^{2} - \frac{1}{900} \right),
    \end{align*}
    represents the exact solution on $ \Omega \times I = (0, 1)^{2} \times \left( 0, 1 \right] $.
    Furthermore, we use the densities $ \bar{\rho} = 0.95 $, $ \rho_{f} = 1 $ and $ \rho_{w} = 2 $.
    For the Lam\'{e} parameters given by
    \begin{align*}
        \lambda & := \frac{\nu E}{(1 + \nu) (1 - 2 \nu)}, &
        \mu & := \frac{E}{2 (1 + \nu)},
    \end{align*}
    we set Young's modulus to $ E = 100 $ and Poisson's ratio to $ \nu = 0.35 $.
    Moreover, we choose $ \alpha = 0.9 $, $ s_{0} = 0.01 $ and $ \bm{K} = \bm{I} $.
    Our initial spatial mesh is a uniform mesh with 50 triangles.
    For the first run, we use a time step size of $ \tau_{0} = 1/10 $ and halve the spatial and temporal mesh sizes after each run.
    
    The numerical results presented in Table~\ref{tab:dg_cgp1} and Table~\ref{tab:dg_cgp2} are for solving problem~\eqref{eq:dg_dyn_biot_velo}
    with the cGP(1) and cGP(2) methods, respectively.
    We observe that the experimental orders of convergence coincide with our theoretical result in Theorem~\ref{thm:dg_discretization_error}.
    
    \begin{table}
        \centering
        \begin{tabular}{cccccccccc}
            \hline
            $ \tau $ & $ h $ & $ \norm{\nabla \bm{e}^{\bm{u}}}_{L^{\infty}(\bm{L}^{2})} $ & eoc & $ \norm{\bm{e}^{\bm{v}}}_{L^{\infty}(\bm{L}^{2})} $ & eoc & $ \norm{\bm{e}^{\bm{w}}}_{L^{\infty}(\bm{L}^{2})} $ & eoc & $ \norm{e^{p}}_{L^{\infty}(L^{2})} $ & eoc \\
            \hline
            $ \tau_{0} / 2^{0} $ & $ h_{0} / 2^{0} $ & $ 7.682 \times 10^{-3} $ & -- & $ 2.781 \times 10^{-3} $ & -- & $ 3.803 \times 10^{-3} $ & -- & $ 1.762 \times 10^{-2} $ & -- \\
            $ \tau_{0} / 2^{1} $ & $ h_{0} / 2^{1} $ & $ 2.099 \times 10^{-3} $ & 1.87 & $ 7.373 \times 10^{-4} $ & 1.92 & $ 1.027 \times 10^{-3} $ & 1.89 & $ 5.539 \times 10^{-3} $ & 1.67 \\
            $ \tau_{0} / 2^{2} $ & $ h_{0} / 2^{2} $ & $ 5.191 \times 10^{-4} $ & 2.02 & $ 2.083 \times 10^{-4} $ & 1.82 & $ 2.609 \times 10^{-4} $ & 1.98 & $ 1.251 \times 10^{-3} $ & 2.15 \\
            $ \tau_{0} / 2^{3} $ & $ h_{0} / 2^{3} $ & $ 1.288 \times 10^{-4} $ & 2.01 & $ 5.213 \times 10^{-5} $ & 2.00 & $ 6.445 \times 10^{-5} $ & 2.02 & $ 3.277 \times 10^{-4} $ & 1.93 \\
            $ \tau_{0} / 2^{4} $ & $ h_{0} / 2^{4} $ & $ 3.197 \times 10^{-5} $ & 2.01 & $ 1.325 \times 10^{-5} $ & 1.98 & $ 1.597 \times 10^{-5} $ & 2.01 & $ 8.130 \times 10^{-5} $ & 2.01 \\
            \hline
            \multicolumn{2}{c}{Optimal} &  & 2.00 &  & 2.00 &  & 2.00 &  & 2.00 \\
            \hline
        \end{tabular}
        \caption{Experimental orders of convergence (eoc) for the cGP(1)-method using $ \ell = 1 $.}
        \label{tab:dg_cgp1}
    \end{table}
    
    \begin{table}
        \centering
        \begin{tabular}{cccccccccc}
            \hline
            $ \tau $ & $ h $ & $ \norm{\nabla \bm{e}^{\bm{u}}}_{L^{\infty}(\bm{L}^{2})} $ & eoc & $ \norm{\bm{e}^{\bm{v}}}_{L^{\infty}(\bm{L}^{2})} $ & eoc & $ \norm{\bm{e}^{\bm{w}}}_{L^{\infty}(\bm{L}^{2})} $ & eoc & $ \norm{e^{p}}_{L^{\infty}(L^{2})} $ & eoc \\
            \hline
            $ \tau_{0} / 2^{0} $ & $ h_{0} / 2^{0} $ & $ 9.466 \times 10^{-4} $ & -- & $ 3.656 \times 10^{-4} $ & -- & $ 1.328 \times 10^{-4} $ & -- & $ 9.725 \times 10^{-4} $ & -- \\
            $ \tau_{0} / 2^{1} $ & $ h_{0} / 2^{1} $ & $ 1.231 \times 10^{-4} $ & 2.94 & $ 4.957 \times 10^{-5} $ & 2.88 & $ 1.846 \times 10^{-5} $ & 2.85 & $ 1.446 \times 10^{-4} $ & 2.75 \\
            $ \tau_{0} / 2^{2} $ & $ h_{0} / 2^{2} $ & $ 1.534 \times 10^{-5} $ & 3.01 & $ 6.424 \times 10^{-6} $ & 2.95 & $ 2.478 \times 10^{-6} $ & 2.90 & $ 2.014 \times 10^{-5} $ & 2.84 \\
            $ \tau_{0} / 2^{3} $ & $ h_{0} / 2^{3} $ & $ 1.898 \times 10^{-6} $ & 3.01 & $ 8.167 \times 10^{-7} $ & 2.98 & $ 3.263 \times 10^{-7} $ & 2.93 & $ 2.748 \times 10^{-6} $ & 2.87 \\
            \hline
            \multicolumn{2}{c}{Optimal} &  & 3.00 &  & 3.00 &  & 3.00 &  & 3.00 \\
            \hline
        \end{tabular}
        \caption{Experimental orders of convergence (eoc) for the cGP(2)-method using $ \ell = 2 $.}
        \label{tab:dg_cgp2}
    \end{table}

    \section{Conclusion}
    In this paper, we proposed a family of time-continuous strongly conservative space-time finite element methods for the dynamic Biot model.
    To this end, we utilized the Galerkin-Petrov scheme for time discretization which constructs a continuous-in-time solution.
    We combined this with a discontinuous Galerkin method for the displacement and a mixed space discretization for the fluid flux, fluid pressure and the temporal derivative of the displacement.
    For this space-time discretization scheme, we derived and presented optimal error estimates in an energy norm.
    Finally, we conducted numerical experiments and observed that the experimental orders of convergence are in agreement with our theoretical findings.

    \section*{Acknowledgments}
    The first, second and third authors acknowledge the funding by the German Science Fund (Deutsche Forschungsgemeinschaft, DFG) as part of the project ``Physics-oriented solvers for multicompartmental poromechanics'' under grant number 456235063.

    \printbibliography
\end{document}